\documentclass{amsart}
\usepackage{amsmath,amsfonts,amssymb,amsthm,amscd,comment,euscript}

\usepackage[all,cmtip]{xy}
\usepackage{graphicx}
\usepackage{tikz}
\usepackage{enumerate} %allows to change enumeration items easily
\definecolor{linkred}{rgb}{0.7,0.2,0.2}
	\definecolor{linkblue}{rgb}{0,0.2,0.6}
	\definecolor{linkgreen}{rgb}{0,0.6,0.2}
\usepackage[colorlinks,plainpages,
    linkcolor=linkblue,
    citecolor=linkgreen,
    urlcolor=linkred]{hyperref}
\usepackage{xcolor}
\usepackage[capitalise]{cleveref}
\crefname{equation}{}{}

\allowdisplaybreaks[4]

\theoremstyle{plain}
\newtheorem{thm}{Theorem}[section]
\newtheorem{lem}[thm]{Lemma}
\newtheorem{cor}[thm]{Corollary}
\newtheorem{prop}[thm]{Proposition}

\theoremstyle{definition}
\newtheorem{example}[thm]{Example}
\newtheorem{rem}[thm]{Remark}
\newtheorem{dfn}[thm]{Definition}

\newcommand{\bbZ}{\mathbb{Z}}
\newcommand{\bbC}{\mathbb{C}}

\newcommand{\bbR}{\mathbb{R}}

\newcommand{\rmp}{\mathrm{p}}

\newcommand{\calS}{\mathcal{S}}

\newcommand{\calO}{\mathcal{O}}
\newcommand{\calI}{\mathcal{I}}
\newcommand{\calT}{\mathcal{T}}

\newcommand{\calF}{\mathcal{F}}

\newcommand{\calB}{\mathcal{B}}

\newcommand{\wt}{\widetilde}

\def\vv{\mathbf{v}}% 
\def\xx{\mathbf{x}}

\newcommand{\al}{\alpha}
\newcommand{\be}{\beta}

\newcommand{\la}{\lambda}

\newcommand{\pt}{{\operatorname{pt}}}

\DeclareMathOperator{\Hom}{Hom}   % the class of point

\newcommand{\aff}{{\operatorname{aff}}}
\newcommand{\ext}{{\operatorname{ext}}}

\newcommand{\calG}{\mathcal{G}}

\DeclareMathOperator{\trace}{\trace}

\newcommand{\id}{\textrm{id}}

\newcommand{\mr}{\mathring}

\DeclareMathOperator{\Frac}{Frac}

\DeclareMathOperator{\Gr}{Gr}
\DeclareMathOperator{\Fl}{Fl}

\DeclareMathOperator{\var}{var}
\DeclareMathOperator{\RHom}{RHom}
\DeclareMathOperator{\MC}{MC}
\DeclareMathOperator{\SMC}{SMC}
\DeclareMathOperator{\rot}{rot}
\DeclareMathOperator{\ev}{ev}
\DeclareMathOperator{\ch}{ch}

\def\c#1{\makebox[0.8pc]{\makebox[0pc]{\(s_{#1}\)}%
    \makebox[0pc]{\(\diagdown\)}}}% 
\def\m#1{\makebox[0.8pc]{\(s_{#1}\)}}% 

\author[C.~Su]{Changjian Su}
\address{Yau Mathematical Sciences Center, Tsinghua University, Beijing, 100084, China}
\email{changjiansu@tsinghua.edu.cn}

\author[R.~Xiong]{Rui~Xiong}
\address{Department of Mathematics and Statistics, University of Ottawa, 150 Louis-Pasteur, Ottawa, ON, K1N 6N5, Canada}
\email{rxiong@uottawa.ca}

\author[C.~Zhong]{Changlong~Zhong}
\address{Department of Mathematics and Statistics, State University of New York at Albany, 1400 Washington Avenue, Albany, NY 12222, USA}
\email{czhong@albany.edu}

\title[Open Projected Richardson Varieties and Affine Schubert Cells]{Motivic Chern Classes of Open Projected Richardson Varieties and of Affine Schubert Cells}

\begin{document}

\date{\today}

\begin{abstract}
The open projected Richardson varieties are images of the open Richardson  varieties of the complete flag variety under the canonical projection to the partial flag variety.  Our main result  compares the Segre motivic Chern (SMC) classes of the open projected Richardson varieties with those of the  affine Schubert cells by pushing or pulling these classes to the affine Grassmannian. The main method is the recursive relation determined by the Demazure--Lusztig operators. As another application of this recursive relation, we relate the localization of the SMC classes to the twisted Kazhdan--Lusztig R-polynomials. In the case of Grassmannians, the open projected Richardson varieties are known as the open positroid varieties. We give a combinatorial formula for the SMC classes of these varieties. 
\end{abstract}

\maketitle

\section{Introduction} 

Let $G$ be a complex connected reductive group with a Borel subgroup and a parabolic subgroup $P\supset B$.  Open (resp. closed) Richardson varieties are defined to be intersections of Schubert cells (resp. Schubert varieties) with opposite Schubert cells (resp. opposite Schubert varieties) in $G/B$, and their projections to $G/P$ are called open (resp. closed) projected Richardson varieties. 
When $G/P$ is a Grassmannian, the projected Richardson varieties are also known as positroid varieties; see \cite{P06,KLS13}. 
Richardson varieties are originally introduced in \cite{KL80}, where the number of points of open Richardson varieties over finite fields is shown to be given by the Kazhdan--Lusztig R-polynomials associated to the Weyl group $W$ of $G$.  Such varieties are latter studied by Richardson \cite{R92} and Deodhar \cite{Deo85}. Many geometric properties of projected Richardson varieties are studied by Knutson, Lam and Speyer \cite{KLS14}. 
Projected Richardson varieties play a crucial role in recent advancements in totally positivity; see \cite{GKL22, BH24} and the references within. 
See \cite{S23} for an excellent survey of this subject, especially from the point of view of Schubert calculus.

Among many studies of (projected) Richardson varieties and their related geometry and combinatorial structures, the following property (along with its special cases and more general cases) plays a central role. Let $\la$ be a dominant coweight whose stabilizer $W_\la< W$ determines a parabolic subgroup $P:=P_\la< G$. Let $t_\la\in W_\ext$ be the translation element in the extended affine Weyl group. Let $W^{P}$ be the set of minimal length representatives of $W/W_P$. In \cite{HL15}, it is shown that there exists an order-preserving embedding 
% there is a {\color{red}order-preserving bijection }
\[
\nu:\mathcal{Q}=\{(u,w)\in W\times W^P\mid u\leq w \}
\to Wt_\la W\subset W_\ext, \quad (u,w)\mapsto ut_\la w^{-1},
\] 
where $\mathcal{Q}$ is the index set of projected Richardson varieties, ordered by closure inclusion.
% where the domain is equipped with certain partial order.
% 
% This combinatorial property is part of a more general framework called Bruhat atlas, see \cite{S10, KWY13, HKL13, H19,H20,  BH21} for more details. 
This combinatorial property was lifted geometrically to Bruhat atlas and Kazhdan--Lusztig atlas \cite{HKL13}; see \cite{H19,BH21} for an introduction. 
% In our setting, it essentially says that there is a stratification of $G/P$ given by open projected Richardson varieties so that $G/P$ is covered by charts $U_f$ with $f\in W^P$ and there is a stratification-preserving isomorphism between $U_f$  and the intersection of certain opposite  Schubert cells with Schubert variety inside the affine flag variety. The study of the Bruhat atlas was extended to more general cases in \cite{BH21}.
 %(see Bao's other work on more general results). 

Following \cite{HL15}, the construction can be stated in more details, and is more relevant to our work. 
Let $\Gr_G=G((z))/G[[z]]$ be the affine Grassmannian and $\Fl_G$ be the affine flag variety.  Let $\Gr_\la$ be the $G[[z]]$-orbit of the element $z^{-\la}G[[z]]/G[[z]]\in \Gr_G$, which is an affine bundle over $G/P$. We have the following diagram of maps
\[
\xymatrix{
G/P \ar[r]^-{i_\la} & \Gr_\la \ar[r]^-{j_\la} & \Gr_G \ar[r]^-r &\Fl_G.
}
\]
Here $i_\la$ is the zero section of the affine bundle, $j_\la$ is the embedding, and $r$ is defined as follows. Let $K\subset G$ be the maximal compact subgroup and $T_\bbR= K\cap T$ be the compact torus. Let $LK$ and $\Omega K$ be the loop group and the based loop group, respectively, then $r$ is the composition
\[
\Gr_G\simeq \Omega K \to LK \to LK/T_\bbR\simeq \Fl_G. 
\]

One can then consider various cohomology and K-theory classes determined by open/closed projected Richardson varieties in $G/P$. For cohomology and K-theory classes, it is proved in \cite{HL15} that the pushforward of  classes of closed projected Richardson varieties via $i_\lambda$
%restricting to the fixed points $v\in W^P$, 
coincide with the pullbacks of the opposite Schubert classes in $\Fl_G$ via the map $q:=r\circ j_\la$, see also \cite{KLS13}. The proof uses the explicit calculation of the restrictions on the two sides, and the AJS--Billey--Graham root formula played a key role. Motivated by this result, \cite{FGSX} compared the Segre--MacPherson classes (a variant form of the Chern--Schwartz--MacPherson (CSM) classes) of open projected Richardson varieties and the Segre--MacPherson classes of opposite affine Schubert cells in $\Fl_G$. 
{The proof is distinct from that in \cite{HL15}.} 
Instead of computing the restriction formulas, recursive formulas are established on both sides, which uniquely determine the classes on the finite and affine sides. One then shows that the two recursive formulas are identical by using the bijection $\nu$. 

The present paper generalizes the result in \cite{FGSX} to the case of the motivic Chern (MC) classes in K-theory. While we still use the inductive method in this paper, the argument is significantly more complicated, due to the quadratic relation in the Hecke algebra (see the table below). For instance, the recursion formulae Theorem \ref{thm:MCPif} (for the finite case) and  Theorem \ref{thm:recSMCGr} (for the affine case) each contains four cases, while the corresponding recursion formula in \cite{FGSX} only has one case.

Note that CSM classes of Schubert cells in homology and MC classes of Schubert cells in K-theory generalize the corresponding cohomology and K-theory Schubert classes and they are closely related to the Maulik--Okounkov stable basis for the Springer resolution $T^*G/B$ \cite{MO,AMSS23,AMSS24,feher2021motivic,Kon}. One simple way of uniformizing these four classes and their corresponding recursion formulas is to use the following table:
\begin{center}
\begin{tabular}{l|l|l}
&permuted by &quadratic relation\\
\hline
Schubert classes in $H_T(G/B)$ & cohomology divided difference operators $\partial_i$&  $\partial_i^2=0$\\
Schubert classes in $K_T(G/B)$ & K-theory divided difference operators $\pi_i$ & $\pi_i^2=\pi_i$\\
CSM classes in $H_T(G/B)$ & cohomology Demazure--Lusztig operators $T_i$ & $T_i^2=1$\\
MC classes in $K_T(G/B)[[y]]$ & Demazure--Lusztig operators $\mathcal{T}_i$ &$\mathcal{T}_i^2=(-y-1)\mathcal{T}_i-y$
\end{tabular}
\end{center}
Based on this observation, our method in this paper can be used to treat all four cases in a uniform way. 

We can now state the main result of this paper. Let $(u,w)\in W\times W^P$ with $u\le w$, and $f:=ut_\la w^{-1}\in W_\ext$. Let $\mr{\Pi}_f$ be the projection in $G/P$ of the open Richardson variety $Y(u)^\circ\cap X(w)^\circ\subset G/B$ where $Y(u)^\circ$ and $X(w)^\circ$ are the opposite Schubert cell and Schubert cell, respectively. Let $\mr{\Sigma}^f$ be the oppsite Schubert cell in $\Fl_G$. We can define the Segre motivic Chern (SMC) classes for $\mr{\Pi}_f$ and $\mr{\Sigma}^f$ (Definitions \ref{def:MCproR},  \ref{def:OpenRic} and   \ref{def:SMCaff}). 
\begin{thm}[\cref{thm:main}]\label{thm-B}
Let $\mathcal{N}$ be the normal bundle of $G/P$ inside $\Gr_\lambda$. Then  
    \[i_{\lambda *}\bigg(\frac{\SMC_y(\mathring{\Pi}_{f})}{\lambda_y(\mathcal{N}^*)}\bigg)=q^*\bigg(\SMC_y(\mathring{\Sigma}^{f})\bigg)\in K_T(\Gr_\lambda)[[y]].\]
\end{thm}
As mentioned above, we use the Demazure--Lusztig operator $\calT_i$ to derive the inductive relation between the two sides. In this process, the so-called left Demazure--Lusztig action on $K_T(G/P)$ and $K_T(\Gr)$ is used. This action was studied in \cite{B97, K03, T09, LZZ20, MNS}. 

The Kazhdan--Lusztig R-polynomials are related to the SMC classes of the opposite Schubert cells, see \cref{prop:smc}. In \cite{GLTW}, the authors introduced twisted R-polynomials. In Section \ref{sec:twistedR}, as another application of the recursive relation  of SMC classes, we show that certain limits of the localization of the SMC classes classes of the opposite Schubert cells give these twisted R-polynomials. Finally, in the case of Grassmannian $\Gr_k(\mathbb{C}^n)$, we give a combinatorial formula for the SMC class of open positroid varieties via pipe dreams.

\subsection*{Acknowledgments}
We would like to thank Thomas Lam, J\"org Sch\"urmann, and David Speyer for helpful conversations. C.S. is supported by the National Key R\&D Program of China (No. 2024YFA1014700). R.X. acknowledges the partial support from the NSERC Discovery grant RGPIN-2022-03060, Canada. C.Z. is partially supported by Simons Foundation Travel Support for Mathematicians TSM-00013828.

\section{Preliminaries}
In this section, we recall the definition of the Kazhdan-Lusztig R-polynomials and some properties of the motivic Chern classes of Schubert cells in the flag variety. 

\subsection{Kazhdan-Lusztig R-polynomials}\label{sec:KLRpoly}

In this subsection, we use $W$ to denote a Coxeter group with length function $\ell(\cdot)$ and Bruhat order $\leq$, for example, a finite Weyl group or an affine Weyl group. Let $H(W)$ be the Hecke algebra over the polynomial ring $\mathbb{Z}[q^{\pm 1}]$. It has basis $\{T_w\mid w\in W\}$, with multiplication defined by $(T_s+1)(T_s-q)=0$ for any simple reflection $s$, and $T_wT_v=T_{wv}$ if $\ell(wv)=\ell(w)+\ell(v)$. Define an involution $h\mapsto \bar{h}$ on $H(w)$ by 
\[\overline{\sum a_wT_w} = \sum \bar{a}_w T_{w^{-1}}^{-1},\]
where $\bar{a}_w$ is the involution on $\mathbb{Z}[q^{\pm 1}]$ defined by $\bar{q}=q^{-1}$.
For any $u,w\in W$, the Kazhdan-Lusztig R-polynomial $R_{u,w}(q)\in \mathbb{Z}[q^{\pm 1}]$ is defined by the formula (see \cite{KL79})
\begin{equation}\label{equ:Rpoly}
    T_{w^{-1}}^{-1}=\sum_{u\leq w} \overline{R_{u,w}(q)} q^{-\ell(u)}T_u,
\end{equation}
and we set $R_{u,w}(q)=0$ if $u\nleq w$. From the definition, we can get the following recursive formulae:
\begin{equation}\label{equ:recR}
    R_{u,w}=\begin{cases}
        R_{su,sw} & \textit{ if } su<u \textit{ and }sw<w;\\
        (q-1)R_{su,w}+qR_{su,sw} & \textit{ if } su>u \textit{ and }sw<w.
    \end{cases}
\end{equation}

If $W$ is finite and $w_0$ is the longest element, 
\begin{equation}\label{equ:Rw0}
    R_{w^{-1}w_0,u^{-1}w_0}(q)=R_{u,w}(q)
\end{equation}
for all $u,w\in W$.

\subsection{Motivic Chern class}\label{sec:MCSchu}
We recall the definition of the motivic Chern classes, following
\cite{brasselet.schurmann.yokura:hirzebruch}. For now on let $X$ be a quasi-projective, complex algebraic variety, with an action of $T$.
The (relative) motivic Grothendieck
group $K_0^T(\var/X)$ of varieties over $X$ is the
quotient of the free abelian group generated by symbols $[f: Z \to X]$ where $Z$ is a quasi-projective $T$-variety and $f: Z \to X$ is a $T$-equivariant morphism modulo the additivity relations
\[
[f: Z \to X] = [f: U \to X] + [f:Z \smallsetminus U \to X]
\]
for $U \subseteq Z$ an open invariant subvariety. For every equivariant
morphism $g:X \to Y$ of quasi-projective $T$-varieties there are 
well-defined push-forwards $g_!: K_0^T(\var/X) \to K_0^T(\var/Y)$ given by composition.

Let $K^T(X)$ be the $T$-equivariant $K$-homology group of $X$, i.e., the Grothendieck group of $T$-equivariant coherent sheaves on $X$, see \cite{chriss2009representation}. Let $y$ be a formal variable. The following theorem was proved in the non-equivariant case by
Brasselet, Sch{\"u}rmann and Yokura \cite[Theorem
2.1]{brasselet.schurmann.yokura:hirzebruch}. Minor changes in the argument are needed in the equivariant case -- see also
\cite{feher2021motivic,AMSS24}. 

\begin{thm}\label{thm:existence}\cite[Theorem 4.2]{AMSS24} 
Let $X$ be a quasi-projective, non-singular, complex algebraic variety
with an action of the torus $T$. There exists a unique natural
transformation $\MC_y: K_0^T(\var/X) \to K^T(X)[y]$ satisfying the
following properties: 
\begin{enumerate} 
\item[(1)] 
It is functorial with respect to push-forwards via $T$-equivariant
proper morphisms of non-singular, quasi-projective varieties. 
		
\item[(2)] 
It satisfies the normalization condition 
\[ 
\MC_y[\id_X: X \to X] = \lambda_y(T^*X) := \sum y^i {[\wedge^i T^*X]}
\in K^T(X)[y] \/. 
\]
\end{enumerate} 
\end{thm}

Let $X$ be a quasiprojective complex manifold. Recall the Serre duality functor $\mathcal{D}(-):=\RHom(-,\omega_{X}^\bullet)$ on $K^T(X)$, where 
$\omega_{X}^\bullet=[\wedge^{\dim X}T^*_X][\dim X]$ is the canonical complex. We extend $\mathcal{D}$ to $K^T(X)[y^{\pm1}]$ by $\mathcal{D}(y^i)=y^{-i}$. 

\begin{dfn}\label{def:MCSMC}
    \begin{enumerate}
        \item If $Y \subseteq X$ is a $T$-invariant subvariety, we call
\[ 
\MC_y(Y) := \MC_y[Y \hookrightarrow X]\in K^T(X)[y] \/ 
\] 
the motivic Chern (MC) class of $Y$, and call
\[ 
\SMC_y(Y) := \frac{\MC_y(Y)}{\lambda_y(T^*X)}\in K^T(X)[[y]] \/ 
\] 
the Segre motivic Chern (SMC) class of $Y$.
\item Let $\Omega \subset X$ be a pure dimensional $T$-stable subvariety of $X$. Let
\[ \widetilde{\MC}_y(\Omega) := (-y)^{\dim \Omega} \frac{\mathcal{D}(\MC_y(\Omega))}{\lambda_y(T^*_X)}\in K^T(X)[[y]]  \/. \] 
    \end{enumerate}
\end{dfn}
\begin{rem}
    Here we are following the convention in \cite{AMSS24b}. Notice that $\widetilde{MC}_y(\Omega)$ is denoted by $\SMC_y(\Omega)$ in \cite{MNS,MS22}.
\end{rem}

\subsection{Motivic Chern classes of Schubert cells}\label{sec:MCSchub}

Let $G$ be a complex connected reductive group with Borel subgroup $B$ and opposite Borel subgroup $B^-$. Let $T=B\cap B^-$ be the maximal torus, and let $W$ be the Weyl group and $\{\alpha_i\mid i\in I\}$ be the set of simple roots.
Assume now that $X=G/B$ is the complete flag variety. Since $X$ is smooth, the $K$-homology group $K^T(G/B)$ can be identified with the $K$-cohomology group $K_T(G/B)$. For any $w\in W$, let $X(w)^\circ:=BwB/B$ and $Y(w)^\circ:=B^-wB/B$ be the Schubert cells, where $B^-$ is the opposite Borel subgroup. Let $X(w):=\overline{BwB/B}$ and $Y(w):=\overline{B^-wB/B}$ be the Schubert varieties.  For any $\mathcal{F}\in K_T(X)$, let \[\chi(X, \mathcal{F})= \int_X \mathcal{F} := \sum_i (-1)^i H^i(X, \mathcal{F}) \in K_T(\pt)=R(T)\/, \]
where $R(T)$ is the representation ring of $T$.
In particular, for $\mathcal{E},\mathcal{F}\in K_T(X)$, we have the following nondegenerate Poincar\'e pairing 
\[ \langle - , - \rangle :K_T(X) \otimes K_T(X) \to K_T(\pt); \quad \langle \mathcal{E},\mathcal{F} \rangle := \int_X \mathcal{E}\otimes\mathcal{F} = \chi(X, \mathcal{E}\otimes\mathcal{F}) \/. \] 

Recall that there is an action of the affine Hecke algebra on $K_T(G/B)[y^{\pm 1}]$, see \cite{lusztig:eqK, MNS}. The left $G$ multiplication action on $G/B$ induces a $W=N_G(T)/T$ action on $K_T(G/B)$. For any $w\in W$, we use $w^L$ to denote this action. These operators are not $K_T(\pt)[y]$-linear, and they act on the base ring $K_T(\pt)[y]$ by the natural action of $W$ on $T$. Then we can define the {\em left} Demazure-Lusztig (DL) operators on $K_T(G/B)[y^{\pm 1}]$:
\begin{equation}\label{equ:lDL}
    \mathcal{T}_i^L: = \frac{1+ye^{-\alpha_i}}{1-e^{-\alpha_i}}s_i^L- \frac{1+y}{1-e^{-\alpha_i}};\quad \mathcal{T}_i^{L,\vee}:=\frac{1+ye^{\alpha_i}}{1-e^{\alpha_i}}s_i^L- \frac{1+y}{1-e^{\alpha_i}}.
\end{equation}
Here $\alpha_i$ is a simple root, and $e^{\alpha_i}$ acts as multiplication by $e^{\alpha_i}$ in the base ring $K_T(\pt)$.
Both the operators $\mathcal{T}_i^L$ and 
$\mathcal{T}_i^{L,\vee}$ satisfy the quadratic relation $(\mathcal{T}_i^{L}+1)(\mathcal{T}_i^{L}+y)=0$ (and same for $\mathcal{T}_i^{L,\vee}$) and the braid relations, see \cite{MNS}. Hence, the $T_w^L$ and $T_w^{L,\vee}$ are well defined for any $w\in W$. For any simple root $\alpha_i$, $-w_0\alpha_i$ is also a simple root. We have 
\begin{equation}\label{equ:TLL}
    w_0^L\mathcal{T}_{-w_0\alpha_i}^L w_0^L=\mathcal{T}_{\alpha_i}^{L,\vee}.
\end{equation}

There is another action using the right DL operators.
For any character $\mu\in X^*(T)$, let $\mathcal{L}_\mu:=G\times_B\mathbb{C}_\mu$ denote the line bundle on $G/B$, where $\mathbb{C}_\mu$ is the one-dimensional representation of $B\rightarrow T \rightarrow GL_1(\mathbb{C}_\mu)$. For any simple root $\alpha_i$, let $P_i$ be the corresponding minimal parabolic subgroup. Let $\pi_i:G/B\rightarrow G/P_i$ be the natural projection. Then we can define the following right Demazure-Lusztig (DL) operators on $K_T(G/B)[y^{\pm 1}]$:
\begin{equation}\label{equ:TiTivee}
    \mathcal{T}^R_i:=(\id+y\mathcal{L}_{\alpha_i})\pi_i^*\pi_{i*}-\id;\quad \mathcal{T}^{R,\vee}_i:=\pi_i^*\pi_{i*}(\id+y\mathcal{L}_{\alpha_i})-\id.
\end{equation}
Here $\mathcal{L}_{\alpha_i}$ acts by tensoring with the line bundle $\mathcal{L}_{\alpha_i}$.
It is proved in \textit{loc. cit.} that these operators $\{\mathcal{T}^R_i\mid i\in I\}$ (resp. $\{\mathcal{T}^{R,\vee}_i\mid i\in I\}$) satisfy the quadratic relation $(\mathcal{T}^R_i+1)(\mathcal{T}^R_i+y)=0$ (resp. $(\mathcal{T}^{R,\vee}_i+1)(\mathcal{T}^{R,\vee}_i+y)=0$) and the braid relation. Hence, for any $w\in W$, $\mathcal{T}^{R}_w$ and $\mathcal{T}^{R,\vee}_w$ are well defined. These operators are $K_T(\pt)[y]$-linear.

The torus $T$-fixed points on the flag variety $G/B$ are $\{\dot{w}B\mid w\in W\}$, where $\dot{w}$ is a lift of $w\in W$ to $G$. For any $\gamma\in K_T(G/B)$ and $w\in W$, we let $\gamma|_w\in K_T(\pt)$ denote the pullback of $\gamma$ to the fixed point $\dot{w}B$. For any simple root $\alpha_i$, the BGG operator $\pi_i^*\pi_{i*}$satisfies
\[\pi_i^*\pi_{i*}(\gamma)|_w=\frac{\gamma|_w-e^{w\alpha_i}\gamma|_{ws_i}}{1-e^{w\alpha_i}}.\]

\begin{prop}\label{prop:mcprop}
    Let $w\in W$, and $s_i$ be a simple reflection. We have the following recursive relations:
 \begin{align}
   \label{eq:TR}     \mathcal{T}^R_i (\MC_y(X(w)^\circ))&=\begin{cases} \MC_y(X(ws_i)^\circ) &\textit{ if } ws_i>w;\\
     -(1+y)\MC_y(X(w)^\circ)-y \MC_y(X(ws_i)^\circ) &\textit{ if } ws_i<w. \end{cases}\\
\label{eq:TL}\mathcal{T}_i^L(\MC_y(X(w)^\circ))&=\begin{cases} \MC_y(X(s_i w)^\circ) &\textit{ if } s_iw>w;\\
     -(1+y)\MC_y(X(w)^\circ)-y \MC_y(X(s_iw)^\circ) &\textit{ if } s_iw<w. \end{cases}\\
\label{eq:TRdual}\calT_i^{R,\vee}(\SMC_y(Y(w)^\circ))&=\begin{cases}
         \SMC_y(Y(ws_i)^\circ) & \textit{ if } ws_i<w;\\
         -(1+y)\SMC_y(Y(w)^\circ)-y\SMC_y(Y(ws_i)^\circ) & \textit{ if } ws_i>w.
     \end{cases}\\
  \label{eq:TLdual}   \mathcal{T}_i^{L,\vee}(\SMC_y(Y(w)^\circ))&=\begin{cases}
           -(1+y)\SMC_y(Y(w)^\circ)-y\cdot\SMC_y(Y(s_iw)^\circ)&\textit{ if } s_iw>w;\\
           \SMC_y(Y(s_iw)^\circ) &\textit{ if } s_iw<w.
       \end{cases}
          \end{align}
     In particular, we have the following relation, and the duality:
     \begin{align}
   \label{eq:MCY}  \widetilde{\MC}_y(Y(w)^{\circ})&=\frac{(-y)^{\dim G/B-\ell(w)}}{\prod_{\alpha>0}(1+ye^{-\alpha})}(\calT_{w_0w}^{R,\vee})^{-1}([\calO_{Y(w_0)}])\/,\\
     \label{eq:SMCY}\SMC_y(Y(w)^\circ)&=\frac{1}{\prod_{\alpha> 0}(1+ye^{-\alpha})}\mathcal{T}^{R,\vee}_{(w_0w)^{-1}}([\calO_{Y(w_0)}]).
     \end{align}
     \begin{align*}
     &\langle \MC_y(X(w)^\circ), \widetilde{\MC}_y(Y(u)^{\circ})\rangle  = \delta_{w,u},\\
     \end{align*}
    %\[\MC_y(X(w)^\circ)=\mathcal{T}^R_{w^{-1}}([\calO_{X(id)}])=\mathcal{T}^L_w([\calO_{X(\id)}]),\]
    % \[ \widetilde{\MC}_y(Y(w)^{\circ})=\frac{(-y)^{\dim G/B-\ell(w)}}{\prod_{\alpha>0}(1+ye^{-\alpha})}(\calT_{w_0w}^{R,\vee})^{-1}([\calO_{Y(w_0)}])\/,\]
    % and
    % \[\SMC_y(Y(w)^\circ,X)=\frac{1}{\prod_{\alpha\geq 0}(1+ye^{-\alpha})}\mathcal{T}^{R,\vee}_{(w_0w)^{-1}}([\calO_{Y(w_0)}]).\]
\end{prop}
\begin{proof}
\cref{eq:TR,eq:TL,eq:MCY} and the duality are proved in \cite{MNS}, while \cref{eq:TRdual,eq:SMCY} are proved in \cite{MS22}. 

We will prove \eqref{eq:TLdual} now. 
    First of all, we have the following formula (see \cite{MNS}),
    \[w_0^L(\MC(Y(w)^\circ))=\MC_y(X(w_0w)^\circ).\]
    The left operators $T_w^L$ commutes with multiplication by $\lambda_y(T^*(G/B))$.
    Hence, by the definition of the SMC class, \cref{equ:TLL} and \cref{prop:mcprop}, we get
    \begin{align*}
       \mathcal{T}_i^{L,\vee}(\SMC_y(Y(w)^\circ))=&w_0^L\mathcal{T}_{-w_0\alpha_i}^L w_0^L(\SMC_y(Y(w)^\circ))\\
       =&\frac{1}{\lambda_y(T^*(G/B))}w_0^L\mathcal{T}_{-w_0\alpha_i}^L w_0^L(\MC(Y(w)^\circ))\\
       =&\frac{1}{\lambda_y(T^*(G/B))}w_0^L\mathcal{T}_{-w_0\alpha_i}^L (\MC(X(w_0w)^\circ))\\
       =&\frac{1}{\lambda_y(T^*(G/B))}\begin{cases}
           -(1+y)w_0^L (\MC(X(w_0w)^\circ))-yw_0^L (\MC(X(w_0s_iw)^\circ))&\textit{ if } s_iw>w;\\
           w_0^L (\MC(X(w_0s_iw)^\circ)) &\textit{ if } s_iw<w.
       \end{cases}
       \\
       =&\begin{cases}
           -(1+y)\SMC_y(Y(w)^\circ)-y\SMC_y(Y(s_iw)^\circ)&\textit{ if } s_iw>w;\\
           \SMC_y(Y(s_iw)^\circ) &\textit{ if } s_iw<w.
       \end{cases}
    \end{align*}
\end{proof}

Using the Kazhdan-Lusztig R-polynomial, we get
\begin{prop}\label{prop:smc}
    For any $w\in W$,
    \[\SMC_y(Y(w)^\circ)=\sum_{u\geq w} R_{w,u}(-y)\cdot \widetilde{\MC}_y(Y(u)^{\circ}).\]
\end{prop}
\begin{proof}
    Since $\{\mathcal{T}_w^{R,\vee}\mid w\in W\}$ satisfy the relations for the Hecke algebra of $W$ with $q=-y$, we get by \cref{equ:Rpoly} and \cref{equ:Rw0},
    \begin{align*}
        \mathcal{T}^{R,\vee}_{w^{-1}w_0}=&\sum_{u\geq w} (-y)^{\ell(u^{-1}w_0)}R_{u^{-1}w_0,w^{-1}w_0}(-y)\cdot (\mathcal{T}^{R,\vee}_{w_0u})^{-1}\\
        =&\sum_{u\geq w} (-y)^{\dim G/B-\ell(u)}R_{w,u}(-y)\cdot (\mathcal{T}^{R,\vee}_{w_0u})^{-1}.
    \end{align*}
    Applying it to the class $[\calO_{Y(w_0)}]$, and using \cref{prop:mcprop}, we get the desired equality.
\end{proof}

\section{Recursion for the SMC classes in the finite case}
In this section, we study the Segre motivic Chern class for the open projected Richardson variety in the partial flag variety. 

\subsection{MC class of open Richardson variety}
Recall that for any $u\leq w$, the open Richardson variety in the complete flag variety is $\mathring{R}_{u,w}:=X(w)^\circ\cap Y(u)^\circ\subset G/B$. We let $\mathring{R}_{u,w}=\varnothing$ if $u\nleq w$.
\begin{dfn}\label{def:MCR}
    We define the MC class of the open Richardson variety $\mathring{R}_{u,w}$ to be 
    \[\MC_y(\mathring{R}_{u,w}):=\MC_y(X(w)^\circ)\otimes \SMC_y(Y(u)^\circ)\in K_T(G/B)[[y]].\]
\end{dfn}
\begin{rem}\label{rem:schu}
    By an unpublished result of Sch\"urmann, which is the K-theoretic analogue of \cite{Sch17}, the right hand side of the above equation equals the MC classes of $\mathring{R}_{u,w}$ inside the full flag variety $G/B$. Hence, the notation is unambiguous.
\end{rem}

Recall the DL operators $\calT_i^{L}, \calT_i^{L,\vee}$ from \cref{equ:lDL}.
\begin{lem}\label{lem:TT}
    For any $\varphi,\psi\in K_T^*(G/B)[[y]]$, we have
    $$(1+y)s_i^L(\varphi\cdot \psi)+(1-e^{-\alpha_i})\cdot s_i^L(\varphi\cdot \calT_i^{L,\vee}(\psi))
    =(1+y)\varphi\cdot\psi+(1-e^{-\alpha_i})\cdot(\calT_i^L(\varphi)\cdot\psi).$$
\end{lem}
\begin{proof}
    By the definition of $\calT_i^L$ and $\calT_i^{L,\vee}$, we have 
    \begin{align*}\calT_i^L(\varphi)\cdot\psi&=\frac{1+ye^{-\alpha_i}}{1-e^{-\alpha_i}}\psi\cdot s_i^L(\varphi)- \frac{1+y}{1-e^{-\alpha_i}}\varphi\cdot\psi,\\
    \varphi\cdot \calT_i^{L,\vee}(\psi)&= \frac{1+ye^{\alpha_i}}{1-e^{\alpha_i}}\varphi\cdot s_i^L(\psi)- \frac{1+y}{1-e^{\alpha_i}}\varphi\cdot\psi.\end{align*}
    The lemma follows from comparison of the two identities. 
\end{proof}
Applying the Lemma to $\varphi=\MC_y(X(w)^\circ)$ and $\psi=\SMC_y(Y(u)^\circ)$, and using \cref{prop:mcprop}, we get
\begin{prop}\label{prop:recforMCB}
For any $u,w\in W$ and simple root $\alpha_i$, we have
\begin{enumerate}
    \item If $s_iw>w$ and $s_iu>u$,
    \begin{align*}
       &(1+y)e^{-\alpha_i}s_i^L(\MC_y(\mathring{R}_{u,w}))-
    (1+y)\MC_y(\mathring{R}_{u,w})\\
    =&y(1-e^{-\alpha_i})\cdot s_i^L(\MC_y(\mathring{R}_{s_iu,w}))+(1-e^{-\alpha_i})\cdot \MC_y(\mathring{R}_{u,s_iw}). 
    \end{align*}
    
    % $$(1+y)s_i^L(\MC_y(\mathring{R}_{u,w}))+(1-e^{-\alpha_i})\cdot s_i^L(-(1+y)\MC_y(\mathring{R}_{u,w})-y\MC_y(\mathring{R}_{s_iu,w}))
    % =(1+y)\MC_y(\mathring{R}_{u,w})+(1-e^{-\alpha_i})\cdot \MC_y(\mathring{R}_{u,s_iw}).$$
    
    \item If $s_iw>w$ and $s_iu<u$,
    \begin{align*}
        &(1+y)s_i^L(\MC_y(\mathring{R}_{u,w}))-
    (1+y)\MC_y(\mathring{R}_{u,w})\\=&(1-e^{-\alpha_i})\cdot\MC_y(\mathring{R}_{u,s_iw})-(1-e^{-\alpha_i})\cdot s_i^L(\MC_y(\mathring{R}_{s_iu,w})).
    \end{align*}

    \item If $s_iw<w$ and $s_iu>u$,
    \begin{align*}
        &(1+y)e^{-\alpha_i}(s_i^L-\id)(\MC_y(\mathring{R}_{u,w}))\\=&y(1-e^{-\alpha_i})\cdot s_i^L(\MC_y(\mathring{R}_{s_iu,w}))-y(1-e^{-\alpha_i})\MC_y(\mathring{R}_{u,s_iw}).
    \end{align*}
  
    \item If $s_iw<w$ and $s_iu<u$,
    \begin{align*}
        &(1+y)(s_i^L-e^{-\alpha_i})(\MC_y(\mathring{R}_{u,w}))\\=&-(1-e^{-\alpha_i})\cdot s_i^L(\MC_y(\mathring{R}_{s_iu,w}))
    -y(1-e^{-\alpha_i})\MC_y(\mathring{R}_{u,s_iw}).
    \end{align*}
  
\end{enumerate}
\end{prop}

\subsection{MC class of the open projected Richardson variety}
Now let $P$ be a parabolic subgroup containing the Borel subgroup $B$, and let $\pi:G/B\rightarrow G/P$ be the natural projection. Let $W_P\subset W$ be the Weyl group determined by $P$, and let $W^P$ be the set of minimal length representatives in the cosets $W/W_P$.  
\begin{dfn}\label{def:MCproR}
    For any $u\in W$, $w\in W^P$, 
    \begin{enumerate}
        \item the open projected Richardson variety is defined by $\mathring{\Pi}_{u,w}=\pi(\mathring{R}_{u,w})$; 
        \item the motivic Chern class of $\mathring{\Pi}_{u,w}$ is defined by
        \[\MC_y(\mathring{\Pi}_{u,w}):=\pi_*\MC_y(\mathring{R}_{u,w}).\]
    \end{enumerate}
\end{dfn}
\begin{rem}
    By the functoriality of $\MC_y$ and \cref{rem:schu}, the class $\pi_*\MC_y(\mathring{R}_{u,w})$ coincides with the usual motivic Chern class of $\mathring{\Pi}_{u,w}$ if we assume Sch\"urmann's K-theoretic intersection formula. Moreover, $\mathring{\Pi}_{u,w}=\varnothing$ if $u\nleq w$.  
\end{rem}

Recall the Demazure--Lusztig operators $\calT_j^R, \calT_j^{R, \vee}$ from \cref{equ:TiTivee}.
\begin{lem}\label{lem:piandTi}
    For any simple reflection $s_j\in W_P$, and $\calF,\calG\in K_T(G/B)[y]$, we have
    \[\pi_*(\mathcal{T}^R_j(\calF)\otimes \calG)=\pi_*(\calF\otimes\mathcal{T}^{R,\vee}_j(\calG)).\]
\end{lem}
\begin{proof}
    Recall $\pi_j$ is the natural projection from $G/B$ to $G/P_j$. Since $s_j\in W_P$, $P_j\subset P$. Let $f:G/P_j\rightarrow G/P$ be the natural projection. Then $\pi=f\circ \pi_j$. Therefore, by the projection formula, we get
    \begin{align*}
        \pi_*(\pi_j^*\pi_{j*}(\calF)\otimes \calG)=&f_*\pi_{j*}(\pi_j^*\pi_{j*}(\calF)\otimes \calG)\\
        =&f_*(\pi_{j*}(\calF)\otimes \pi_{j*}(\calG))\\
        =&f_*\pi_{j*}(\calF\otimes \pi_j^{*}\pi_{j*}(\calG))\\
        =&\pi_*(\calF\otimes \pi_j^{*}\pi_{j*}(\calG)).
    \end{align*}
    Then the lemma follows from this and the definitions of $\mathcal{T}^R_j$ and $\mathcal{T}_j^{R,\vee}$.
\end{proof}
Using this lemma, we get the following formula. 
\begin{cor}\label{cor:pinotminimal}
    For any $u\in W$, $w\in W^P$, $s_j\in W_P$, we have
    \[
    \pi_*(\MC_y(\mathring{R}_{u,ws_j}))=\begin{cases}
        \MC_y(\mathring{\Pi}_{us_j,w}) & \textit{ if } us_j<u;\\
        -(1+y)\MC_y(\mathring{\Pi}_{u,w})-y\MC_y(\mathring{\Pi}_{us_j,w}) & \textit{ if } us_j>u.
        \end{cases}\]
\end{cor}
\begin{proof}
    Since $w\in W^P$ and $s_j\in W_P$, $ws_j>w$. By the definition of MC classes of open Richardson and open projected Richardson varieties, \cref{lem:piandTi}, and \cref{prop:mcprop}(4), we get
    \begin{align*}
        &\pi_*(\MC_y(\mathring{R}_{u,ws_j}))\\
        =&\pi_*\bigg(\MC_y(X(ws_j)^\circ)\otimes \SMC_y(Y(u)^\circ)\bigg)\\
        =&\pi_*\bigg(\mathcal{T}_j^R(\MC_y(X(w)^\circ))\otimes \SMC_y(Y(u)^\circ)\bigg)\\
        =&\pi_*\bigg(\MC_y(X(w)^\circ)\otimes\mathcal{T}_j^{R,\vee}(\SMC_y(Y(u)^\circ))\bigg)\\
        =&\begin{cases}
        \pi_*\bigg(\MC_y(X(w)^\circ)\otimes\SMC_y(Y(us_j)^\circ)\bigg) & \textit{ if } us_j<u;\\
        -(1+y)\pi_*\bigg(\MC_y(X(w)^\circ)\otimes\SMC_y(Y(u)^\circ)\bigg)-y\pi_*\bigg(\MC_y(X(w)^\circ)\otimes\SMC_y(Y(us_j)^\circ)\bigg) & \textit{ if } us_j>u
        \end{cases}\\
        =&\begin{cases}
        \pi_*(\MC_y(\mathring{R}_{us_j,w})) & \textit{ if } us_j<u;\\
        -(1+y)\pi_*(\MC_y(\mathring{R}_{u,w}))-y\pi_*(\MC_y(\mathring{R}_{us_j,w})) & \textit{ if } us_j>u
        \end{cases}\\
        =&\begin{cases}
        \MC_y(\mathring{\Pi}_{us_j,w}) & \textit{ if } us_j<u;\\
        -(1+y)\MC_y(\mathring{\Pi}_{u,w})-y\MC_y(\mathring{\Pi}_{us_j,w}) & \textit{ if } us_j>u.
        \end{cases}
    \end{align*}
\end{proof}

We also need the following result of Deodhar.
\begin{lem}\cite[Lemma 2.1]{Deo87}\label{lem:parabolic}
For any simple root $\alpha_i$ and $w\in W^P$, exactly one of the following cases hold:
\begin{enumerate}
\item $s_i w<w$, in which case $s_i w\in W^P$;
\item $s_i w>w$ and $s_i w\in W^P$;
\item $s_i w>w$ and $s_i w=ws_j$ for some simple reflection $s_j\in W_P$. 
\end{enumerate}
\end{lem}

Since $G$ acts on $G/P$, we can also define the action of the Weyl group $W$ on $K_T(G/P)[y]$, denoted by $w^L$. Since the projection $\pi$ commutes with the $G$-action, $w^L$ commutes with the pushforward $\pi_*$. Applying $\pi_*$ to the equations in \cref{prop:recforMCB}, and using \cref{lem:parabolic} and \cref{cor:pinotminimal}, we get the following recursion formuale for the MC classes of the open projected Richardson varieties. 
\begin{cor}\label{cor:recMCP}
    For any $u\in W$, $w\in W^P$, and simple root $\alpha_i$, we have
\begin{enumerate}
    \item If $s_iw>w$, $s_iw\in W^P$, and $s_iu>u$,
    \begin{align*}
        &(1+y)e^{-\alpha_i}s_i^L(\MC_y(\mathring{\Pi}_{u,w}))-
    (1+y)\MC_y(\mathring{\Pi}_{u,w})\\=&y(1-e^{-\alpha_i})\cdot s_i^L(\MC_y(\mathring{\Pi}_{s_iu,w}))+(1-e^{-\alpha_i})\cdot \MC_y(\mathring{\Pi}_{u,s_iw}).
    \end{align*}

    \item If $s_iw>w$, $s_iw=ws_j$ for some simple reflection $s_j\in W$, $s_iu>u$, and $us_j<u$,
    \begin{align*}
        &(1+y)e^{-\alpha_i}s_i^L(\MC_y(\mathring{\Pi}_{u,w}))-
    (1+y)\MC_y(\mathring{\Pi}_{u,w})\\=&y(1-e^{-\alpha_i})\cdot s_i^L(\MC_y(\mathring{\Pi}_{s_iu,w}))+(1-e^{-\alpha_i})\cdot \MC_y(\mathring{\Pi}_{us_j,w}).
    \end{align*}

    \item If $s_iw>w$, $s_iw=ws_j$ for some simple reflection $s_j\in W$, $s_iu>u$, and $us_j>u$ 
    \begin{align*}
        &(1+y)e^{-\alpha_i}(s_i^L-\id)(\MC_y(\mathring{\Pi}_{u,w}))\\
    =&y(1-e^{-\alpha_i})\cdot s_i^L(\MC_y(\mathring{\Pi}_{s_iu,w}))-y(1-e^{-\alpha_i})\cdot \MC_y(\mathring{\Pi}_{us_j,w}).
    \end{align*}

    \item If $s_iw>w$, $s_iw\in W^P$, and $s_iu<u$,
    \begin{align*}
        &(1+y)s_i^L(\MC_y(\mathring{\Pi}_{u,w}))-
    (1+y)\MC_y(\mathring{\Pi}_{u,w})\\=&(1-e^{-\alpha_i})\cdot\MC_y(\mathring{\Pi}_{u,s_iw})-(1-e^{-\alpha_i})\cdot s_i^L(\MC_y(\mathring{\Pi}_{s_iu,w})).
    \end{align*}

    \item If $s_iw>w$, $s_iw=ws_j$ for some simple reflection $s_j\in W_P$, $s_iu<u$, and $us_j<u$,
    \begin{align*}
        &(1+y)s_i^L(\MC_y(\mathring{\Pi}_{u,w}))-
    (1+y)\MC_y(\mathring{\Pi}_{u,w})\\=&(1-e^{-\alpha_i})\cdot\MC_y(\mathring{\Pi}_{us_j,w})-(1-e^{-\alpha_i})\cdot s_i^L(\MC_y(\mathring{\Pi}_{s_iu,w}))
    \end{align*}

    \item If $s_iw>w$, $s_iw=ws_j$ for some simple reflection $s_j\in W_P$, $s_iu<u$, and $us_j>u$,
    \begin{align*}
        &(1+y)s_i^L(\MC_y(\mathring{\Pi}_{u,w}))-
    (1+y)e^{-\alpha_i}\MC_y(\mathring{\Pi}_{u,w})\\
    =&-y(1-e^{-\alpha_i})\cdot\MC_y(\mathring{\Pi}_{us_j,w})-(1-e^{-\alpha_i})\cdot s_i^L(\MC_y(\mathring{\Pi}_{s_iu,w})).
    \end{align*}
   
    \item If $s_iw<w$ and $s_iu>u$,
    \begin{align*}
        &(1+y)e^{-\alpha_i}(s_i^L-\id)(\MC_y(\mathring{\Pi}_{u,w}))\\=& y(1-e^{-\alpha_i})\cdot s_i^L(\MC_y(\mathring{\Pi}_{s_iu,w}))-y(1-e^{-\alpha_i})\MC_y(\mathring{\Pi}_{u,s_iw}).
    \end{align*}

    \item If $s_iw<w$ and $s_iu<u$,
    \begin{align*}
        &(1+y)(s_i^L-e^{-\alpha_i})(\MC_y(\mathring{\Pi}_{u,w}))\\=&-(1-e^{-\alpha_i})\cdot s_i^L(\MC_y(\mathring{\Pi}_{s_iu,w}))
    -y(1-e^{-\alpha_i})\MC_y(\mathring{\Pi}_{u,s_iw}).
    \end{align*}
\end{enumerate}
\end{cor}

Now, let us rewrite the recursion in terms of the extended affine Weyl group $W_{\ext}:=W\ltimes X_*(T)$, where $X_*(T)$ is the cocharacter lattice. 
In the remaining part of this section, let $\lambda$ be a dominant cocharacter whose stabilizer is $W_P$. For any $u,w\in W$, let \[f^{\lambda}_{u,w}:=ut_\lambda w^{-1}\in W_{\ext}.\]
Let us denote $\mathcal{B}^+=\{f^{\lambda}_{u,w}\mid u\in W, w\in W^P\}$. By \cite{HL15}, there is a bijection
\[W\times W^P\rightarrow \mathcal{B}^+, \quad (u,w)\mapsto f^{\lambda}_{u,w}. \]
Moreover, for any $w\in W^P$, $t_\lambda w^{-1}$ is a minimal length representative in the right coset $W\cdot  t_\lambda w^{-1}\subset W_{\ext}$. In particular, for any $(u,w)\in W\times W^P$, 
\begin{equation}\label{equ:leng}
\ell(ut_\lambda w^{-1})=\ell(u)+\ell(t_\lambda)-\ell(w). 
\end{equation} 
Using these, it is easy to get
\begin{lem}\label{lem:fuw}
    Fix $u\in W, w\in W^P$, and let $f=ut_\lambda w^{-1}$. For any simple root $\alpha_i$, we have
    \begin{enumerate}
        \item $s_if<f\Leftrightarrow s_iu<u$.
        \item $s_if>f\Leftrightarrow s_iu>u$.
        \item $fs_i<f\Leftrightarrow w<s_iw\in W^P$, or $s_iw=ws_j$ for some simple reflection $s_j\in W_P$ and $us_j<u$.
        \item $fs_i>f\Leftrightarrow s_iw<w$, or $s_iw=ws_j$ for some simple reflection $s_j\in W_P$ and $us_j>u$.
    \end{enumerate}
    Moreover, if $s_iw=ws_j$ for some simple reflection $s_j\in W_P$, then $us_jt_\lambda w^{-1}=ut_\lambda (s_iw)^{-1}=fs_i$.
\end{lem}

\begin{dfn}\label{def:OpenRic}
    For $u\in W, w\in W^P$,  denote by $\mathring{\Pi}_{f^{\lambda}_{u,w}}=\mathring{\Pi}_{u,w}\subset G/P$
     the open projected Richardson variety. Define the SMC class of $\mathring{\Pi}_{f^{\lambda}_{u,w}}$ by \[\SMC_y(\mathring{\Pi}_{f^{\lambda}_{u,w}}):=\frac{\MC_y(\mathring{\Pi}_{f^{\lambda}_{u,w}})}{\lambda_y(T^*(G/P))}\in K_T(G/P)[[y]].\]
\end{dfn}
Since the class $\lambda_y(T^*(G/P))$ is $G$-equivariant, it is invariant under the action $w^L$ for any $w\in W$. Therefore, using \cref{lem:fuw}, we can reformulate \cref{cor:recMCP} in the following way. 
\begin{thm}\label{thm:MCPif}
    Assume $f=f^{\lambda}_{u,w}$ for some $u\in W, w\in W^P$. For any simple root $\alpha_i$, we have
    \begin{enumerate}
        \item If $s_if<f$ and $fs_i<f$, then
        \begin{align*}
        &(1+y)\cdot s_i^L(\SMC_y(\mathring{\Pi}_f))-
    (1+y)\cdot \SMC_y(\mathring{\Pi}_f)\\=&(1-e^{-\alpha_i})\cdot\SMC_y(\mathring{\Pi}_{fs_i})-(1-e^{-\alpha_i})\cdot s_i^L(\SMC_y(\mathring{\Pi}_{s_if})).
    \end{align*}
    \item If $s_if<f$ and $fs_i>f$, then
    \begin{align*}
        &(1+y)\cdot s_i^L(\SMC_y(\mathring{\Pi}_f))-(1+y)e^{-\alpha_i}\cdot \SMC_y(\mathring{\Pi}_f)\\=&-(1-e^{-\alpha_i})\cdot s_i^L(\SMC_y(\mathring{\Pi}_{s_if}))
    -y(1-e^{-\alpha_i})\cdot \SMC_y(\mathring{\Pi}_{fs_i}).
    \end{align*}
    \item If $s_if>f$ and $fs_i<f$, then
    \begin{align*}
        &(1+y)e^{-\alpha_i}\cdot s_i^L(\SMC_y(\mathring{\Pi}_{f}))-
    (1+y)\cdot \SMC_y(\mathring{\Pi}_{f})\\=&y(1-e^{-\alpha_i})\cdot s_i^L(\SMC_y(\mathring{\Pi}_{s_if}))+(1-e^{-\alpha_i})\cdot \SMC_y(\mathring{\Pi}_{fs_i}).
    \end{align*}
    
    \item If $s_if>f$ and $fs_i>f$, then
    \begin{align*}
        &(1+y)e^{-\alpha_i}\cdot s_i^L(\SMC_y(\mathring{\Pi}_{f}))-(1+y)e^{-\alpha_i}\cdot (\SMC_y(\mathring{\Pi}_{f}))\\
    =&y(1-e^{-\alpha_i})\cdot s_i^L(\SMC_y(\mathring{\Pi}_{s_if}))-y(1-e^{-\alpha_i})\cdot \SMC_y(\mathring{\Pi}_{fs_i}).
    \end{align*}
    \end{enumerate}
\end{thm}
\begin{proof}
    Let us only prove the first one, as the others can be checked similarly. By \cref{lem:fuw}, the case $s_if<f$ and $fs_i>f$ corresponds to the cases (4) and (5) in \cref{cor:recMCP}, both of which give the same recursion:
    \begin{align*}
        &(1+y)\cdot s_i^L(\SMC_y(\mathring{\Pi}_f))-
    (1+y)\cdot \SMC_y(\mathring{\Pi}_f)\\=&(1-e^{-\alpha_i})\cdot\SMC_y(\mathring{\Pi}_{fs_i})-(1-e^{-\alpha_i})\cdot s_i^L(\SMC_y(\mathring{\Pi}_{s_if})).
    \end{align*}
    %%%%% (2) corresponds to cases 6 and 8; (3) corresponds to 1 and 2, while (4) corresponds to 3 and 7.
\end{proof}

\section{SMC classes in the affine flag variety}
In this section, we study the motivic Chern classes in the affine flag variety. 

\subsection{Affine flag variety}
Let $X_*(T)$ be the cocharacter lattice of $T$, and let 
$$W_{\ext}:=W\ltimes X_*(T)$$ be the extended affine Weyl group.
For a cocharacter $\lambda\in X_*(T)$, we denote by $t_\lambda$ the corresponding element in $W_{\ext}$. Note that in $W_{\ext}$, we have 
$$wt_\lambda w^{-1}=t_{w\lambda},\qquad w\in W,\lambda\in X_*(T).$$
Denote the coroot lattice as $Q^\vee\subseteq X_*(T)$. 
It is known that the subgroup $W_\aff:=W\ltimes Q^\vee$, called the affine Weyl group, is the Coxeter group of the corresponding affine Dynkin diagram with generators
$$s_i\in W\, (i\in I),\quad\text{ and \quad} s_0=t_{\theta^\vee}s_\theta,
$$
where $\theta$ is the highest root. 
Let $R$ be the set of roots, with positive (resp. negative) roots denoted by $R^+$ (resp., $R^-$). For any $\alpha\in R$, we use $\alpha>0$ to denote that $\alpha\in R^+$.
The length function on $W_\aff$ can be extended to $W_{\ext}$:
$$\ell(wt_\lambda)
=
\sum_{\alpha>0,w\alpha>0}\big|\langle \alpha,\lambda\rangle\big|+
\sum_{\alpha>0,w\alpha<0}\big|\langle \alpha,\lambda\rangle+1\big|.
$$ 
Let $\Omega\subset W_{\ext}$ be the set of length 0 elements. Then $W_{\ext} =W_\aff\rtimes \Omega$, and we can also extend the Bruhat order on $W_\aff$ to $W_\ext$ by $w\gamma\leq w'\gamma'$ if and only if $w\leq w'$ and $\gamma=\gamma'$, where $w,w'\in W_\aff$ and  $\gamma,\gamma'\in \Omega$. With this length function and Bruhat order, all the constructions in \cref{sec:KLRpoly} can be extended to $W_\ext$.

Let $\mathbb{C}[[z]]$ (resp. $\mathbb{C}(\!(z)\!):=\mathbb{C}[[z]][z^{-1}]$) be the formal power series ring (resp. formal Laurent series ring), and let $G[[z]]$ (resp. $G(\!(z)\!)$)  be the $\mathbb{C}[[z]]$-points (resp. $\mathbb{C}(\!(z)\!)$-points) of $G$. There is an evaluation at $z=0$ map from $G[[z]]$ to $G$, and let $\mathcal{I}$ (resp. $\calI^-$) be the inverse image of the Borel subgroup $B$ (resp., the opposite Borel subgroup $B^-$).  Then the \emph{affine flag variety} is
$$\Fl_G = G(\!(z)\!)/\mathcal{I},$$
whose $T$-fixed points
$(\Fl_G)^T$ can be identified with  $W_{\ext}$ as follows. Any cocharacter $\lambda\in X_*(T)$ defines a morphism $\mathbb{C}(\!(z)\!)^*\rightarrow T(\!(z)\!)\subset G(\!(z)\!)$, we use $z^\lambda\in G(\!(z)\!)$ to denote the image of $z$. For any $w\in W$, let $\dot{w}$ denote a lift of it in $G[[z]]$. Then for each $wt_\lambda\in W_{\ext}$, the corresponding fixed point in $\Fl_G$ is $\dot{w}z^{-\lambda}\mathcal{I}\in \Fl_G$, which will be just denoted by $wt_\lambda$ throughout the paper. Let $\mathring{\Sigma}_{wt_\lambda}:=\mathcal{I}\dot{w}z^{-\lambda}\mathcal{I}/\mathcal{I}$ be the {\it affine Schubert cell} of dimension $\ell(wt_\lambda)$. The  affine flag variety has a cell decomposition 
\[\Fl_G=\bigsqcup_{wt_\lambda\in W_{\ext}}\mathring{\Sigma}_{wt_\lambda}.\]
For each $n\geq 0$, let 
\begin{equation}\label{equ:ind}
    X_n:=\bigsqcup_{wt_\lambda\in W_{\ext},\,  \ell(wt_\lambda)\leq n}\mathring{\Sigma}_{wt_\lambda}. 
\end{equation}
Then $\Fl_G$ is the increasing limit of $X_n$, which is called an ind-variety. Let   $\mathring{\Sigma}^{wt_\la}:=\calI^-\dot{w}z^{-\lambda}\calI/\calI$ be the opposite affine Schubert cell.

There is a loop rotation action of $\mathbb{C}^*_{\rot}:=\mathbb{C}^*$ on $\Fl_G$ by scaling the parameter $z$. Let $\delta$ be the degree one character of this action, which is also the imaginary root for the corresponding Kac--Moody Lie algebra $\hat{\mathfrak{g}}$. Then the positive (real)  affine roots are 
\begin{equation}\label{equ:affpos}
    R^+_{\aff}=\{\alpha+k\delta\mid \alpha\in R^+, k\geq 0 \text{ or } \alpha\in R^-, k\geq 1\},
\end{equation}
and the negative (real) affine roots are 
\[R^-_{\aff}=\{\alpha+k\delta\mid \alpha\in R^+, k\leq -1 \text{ or } \alpha\in R^-, k\leq 0\}.\]
The extended affine Weyl group $W_{\ext}$ acts on the lattice $X^*(T)\oplus\mathbb{Z}\delta$ by the following formula (see \cite[(3.1)]{LS10})
\begin{equation}\label{equ:affact}
    wt_\lambda(\mu+k\delta)=w(\mu)+(k-\langle{\lambda, \mu\rangle})\delta, ~\mu\in X^*(T). 
\end{equation}
 We need the following lemma computing the torus weights of the tangent space, see \cite{FGSX} for a proof.
\begin{lem}\label{lem:tangentwt}
    For any dominant cocharacter $\lambda$, the torus $T\times \mathbb{C}^*_{\rot}$ weights of the tangent space $T_{t_\lambda}(\mathring{\Sigma}_{t_\lambda})$ are 
    \[\{\alpha+k\delta\mid \alpha\in R^-, 1\leq k\leq -\langle\lambda,\alpha\rangle\}.\]
\end{lem}

\subsection{MC classes in the affine flag variety}
Recall $T\subset G$ is the maximal torus in $G$. We consider the $T$-equivariant K-homology group of $\Fl_G$, which is defined by 
\[K^T(\Fl_G):=\lim_{n\rightarrow\infty}K^T(X_n)\]
using the ind-scheme structure \eqref{equ:ind}.

On the other hand, $\Fl_G$ can be realized as a Kac--Moody flag variety $\hat{G}/\hat{B}$ for the corresponding affine Kac-Moody group $\hat{G}$, which is an ind-finite ind-variety with a stratification by the finite-dimensional Schubert cells \cite{Ku02}. The Weyl group of the affine Kac--Moody group $\hat{G}$ is $W_\ext$. Hence, as in \cref{sec:MCSchu}, we have the operators $w^L$ and $w^R$ on $K_T(\Fl_G)$ for any $w\in W_\ext$. Let $I_{\aff}:=I\sqcup\{0\}$ be the vertices of the affine Dynkin diagram. Then the operator $\mathcal{T}_i^L, \mathcal{T}_i^{L,\vee}, \mathcal{T}_i^R, \mathcal{T}_i^{R,\vee}$ for $i\in I_{\aff}$ in \cref{equ:lDL,equ:TiTivee} can also be defined for $K^T(\Fl_G)$, with $\alpha_0=-\theta$
as we only consider the small torus $T$ instead of the affine torus in $\hat{G}$, where $\theta$ is the highest root. These operators satisfy the same relations as before.

Recall that for any $f\in W_\ext$, $\mathring{\Sigma}_{f}\subset \Fl_G$ denotes the finite dimensional affine Schubert cell. The MC class $\MC_y(\mathring{\Sigma}_{f})\in K^T(\Fl_G)[y]$ is well defined as $\Fl_G$ is an ind-scheme, and they form a basis for the localized equivariant K-homology group. Since the Schubert variety $\Sigma_f:=\overline{\mathring{\Sigma}_{f}}$ also has a Bott--Samelson resolution as in the finite type case (see \cite{Ku02}), the proof of \cite[Theorem 1.1]{AMSS24} also works for the affine flag variety $\Fl_G$. Hence, for any $i\in I_{\aff}$, we have the following formula similar to \eqref{eq:TR}:
\begin{equation}\label{equ:TiRaff}    \mathcal{T}_i^R(\MC_y(\mathring{\Sigma}_{f}))=\begin{cases} \MC_y(\mathring{\Sigma}_{fs_i}) &\textit{ if } fs_i>f;\\
     -(1+y)\MC_y(\mathring{\Sigma}_{f})-y \MC_y(\mathring{\Sigma}_{fs_i}) &\textit{ if } fs_i<f. \end{cases}
\end{equation}

Moreover, as the proof in \cite[Theorem 7.6]{MNS} only depends on the above formula and a computation for $\mathbb{P}^1$, we get the affine version of \eqref{eq:TL}:
\begin{equation}\label{equ:TiLaff}
    \mathcal{T}_i^L(\MC_y(\mathring{\Sigma}_{f}))=\begin{cases}
    \MC_y(\mathring{\Sigma}_{s_if}), &\textit{ if } s_if>f;\\
    -(1+y)\MC_y(\mathring{\Sigma}_{f})-y \MC_y(\mathring{\Sigma}_{s_if}) &\textit{ if } s_if<f. 
\end{cases}
\end{equation}

Recall that the $T$-equivariant K-cohomology of the affine flag variety $K_T(\Fl_G)$ can be identified with $\Hom_{K^T(\pt)}(K^T(\Fl_G),K^T(\pt))$, see \cite{LSS10}. Hence, there is a nondegenerate pairing
\[\langle-,-\rangle:K^T(\Fl_G)\times K_T(\Fl_G)\rightarrow K_T(\pt).\]
As in the finite type case in \cref{prop:mcprop}, we define the classes $\widetilde{\MC}_y(\mathring{\Sigma}^{f})\in K_T(\Fl_G)[[y]]$  associated with the opposite Schubert cells $\mathring{\Sigma}^{f}$ to be the dual basis of the MC classes:
\begin{equation}\label{equ:defdual}
    \langle \MC_y(\mathring{\Sigma} _f), \widetilde{\MC}_y(\mathring{\Sigma}^{g})\rangle=\delta_{f,g}, ~f,g\in W_\ext. 
\end{equation}
\begin{rem}
    To be more precise, we need to consider the thick affine flag variety $\widetilde{\Fl}_G$, and the SM class is defined in the K-cohomology ring $K_T(\widetilde{\Fl}_G)$, which is isomorphic to $\Hom_{K^T(\pt)}(K^T(\Fl_G),K^T(\pt))$, see \cite{LSS10}.
\end{rem}
By the similar calculations as in \cite{MNS}, we can show that for any $\gamma_1\in K^T(\Fl_G)$ and $\gamma_2\in K_T(\Fl_G)$
\begin{equation}\label{equ:adj}
    \langle \mathcal{T}_i^{R}(\gamma_1), \gamma_2\rangle = \langle\gamma_1, \mathcal{T}_i^{R,\vee}(\gamma_2)\rangle, \text{\quad and \quad} \langle \mathcal{T}_i^{L}(\gamma_1), \gamma_2\rangle = s_i\cdot \langle\gamma_1, \mathcal{T}_i^{L,\vee}(\gamma_2)\rangle.
\end{equation}
Combining \cref{equ:TiRaff,equ:TiLaff,equ:defdual,equ:adj}, we get
\begin{align}\label{equ:TiLveeaff}
     \mathcal{T}_i^{L,\vee} (\widetilde{\MC}_y(\mathring{\Sigma}^{f}))&=\begin{cases}
        -y\widetilde{\MC}_y(\mathring{\Sigma}^{s_if}) &\qquad \textit{ if } s_if>f ;\\
        \widetilde{\MC}_y(\mathring{\Sigma}^{s_if})-(1+y)\widetilde{\MC}_y(\mathring{\Sigma}^{f}) &\qquad \textit{ if } s_if<f. 
        \end{cases}\\
    \mathcal{T}_i^{R,\vee} (\widetilde{\MC}_y(\mathring{\Sigma}^{f}))&=\begin{cases}
        -y\widetilde{\MC}_y(\mathring{\Sigma}^{fs_i}) &\qquad \textit{ if } fs_i>f ;\\
        \widetilde{\MC}_y(\mathring{\Sigma}^{fs_i})-(1+y)\widetilde{\MC}_y(\mathring{\Sigma}^{f}) &\qquad \textit{ if } fs_i<f. 
    \end{cases}
\end{align}

Motivated by \cref{prop:smc}, we define the Segre motivic Chern class of the opposite affine Schubert cell $\mathring{\Sigma}^{f}$ as follows.
\begin{dfn}\label{def:SMCaff}
    For any $w\in W$, let
    \[\SMC_y(\mathring{\Sigma}^{f})=\sum_{g\in W_\ext} R_{f,g}(-y)\cdot \widetilde{\MC}_y(\mathring{\Sigma}^{g})\in K_T(\Fl_G)[[y]],\]
    where $R_{f,g}(-y)$ is the R-polynomial for the Hecke algebra (with $y=-q$) associated with the extended affine Weyl group $W_\ext$.
\end{dfn}
\begin{rem}
    Since $R_{f,g}(-y)=0$ if $f\nleq g$, the above is a summation over $g$ satisfying $f\leq g$. Although it is an infinite sum, the value of it on $\MC_y(\mathring{\Sigma} _g)$ is $R_{f,g}(-y)$. Hence, it is a well-defined class in $K_T(\Fl_G)[[y]]$.
\end{rem}
We have the following recursive formula for these classes.
\begin{thm}\label{thm:SMCrecaff}
    For any simple root $\alpha_i$ with $i\in I_\aff$, and $f\in W_\ext$,
\begin{align}\mathcal{T}^{L,\vee}_i(\SMC_y(\mathring{\Sigma}^{f}))&=\begin{cases}
           -(1+y)\cdot \SMC_y(\mathring{\Sigma}^{f})-y\cdot\SMC_y(\mathring{\Sigma}^{s_if})&\textit{ if } s_if>f;\\
           \SMC_y(\mathring{\Sigma}^{s_if}) &\textit{ if } s_if<f,
       \end{cases}\\
       \calT_i^{R,\vee}(\SMC_y(\mathring{\Sigma}^{f}))&=\begin{cases}
         \SMC_y(\mathring{\Sigma}^{fs_i}) & \textit{ if } fs_i<f;\\
         -(1+y)\cdot \SMC_y(\mathring{\Sigma}^{f})-y\cdot \SMC_y(\mathring{\Sigma}^{fs_i}) & \textit{ if } fs_i>f.
     \end{cases}\end{align}
\end{thm}
\begin{proof}
    Let us prove the first one, as the second one can be checked similarly. By the quadratic relation $(\mathcal{T}^{L,\vee}_i+1)(\mathcal{T}^{L,\vee}_i+y)=0$, it is enough to prove the case $s_if<f$, which we now assume. Then by \cref{equ:TiLveeaff}, we get
    \begin{align*}
        &\mathcal{T}^{L,\vee}_i(\SMC_y(\mathring{\Sigma}^{f}))\\
        =&\sum_{s_ig>g} -y\cdot R_{f,g}(-y)\cdot \widetilde{\MC}_y(\mathring{\Sigma}^{s_ig})+\sum_{s_ig<g} R_{f,g}(-y)\bigg(\widetilde{\MC}_y(\mathring{\Sigma}^{s_ig})-(1+y)\widetilde{\MC}_y(\mathring{\Sigma}^{g})\bigg)\\
        =&\sum_{g>s_ig} \bigg(-y\cdot R_{f,s_ig}(-y)-(1+y)\cdot R_{f,g}(-y)\bigg) \cdot\widetilde{\MC}_y(\mathring{\Sigma}^{g})+\sum_{g<s_ig} R_{f,s_ig}(-y)\cdot\widetilde{\MC}_y(\mathring{\Sigma}^{g})\\
        =&\sum_{g>s_ig} R_{s_if,g}(-y)\cdot \widetilde{\MC}_y(\mathring{\Sigma}^{g})+\sum_{g<s_ig} R_{s_if,g}(-y)\cdot \widetilde{\MC}_y(\mathring{\Sigma}^{g})\\
        =&\SMC_y(\mathring{\Sigma}^{s_if}),
    \end{align*}
    where the third equality follows from \cref{equ:recR}.
\end{proof}

The map $W_\ext\ni g\mapsto gI\in\Fl_G$ defines a bijection between the torus fixed points in $\Fl_G$ and $W_\ext$. For any $g\in W_\ext$, we let $[\mathcal{O}_g]\in K^T(\Fl_G)$ denote the class of the structure sheaf of the fixed point $gI\in \Fl_G$. Then $\{[\mathcal{O}_g]\mid g\in W_\ext\}$ forms a basis for the localized group $K^T(\Fl_G)\otimes_{K_T(\pt)}\Frac K_T(\pt)$, where $\Frac K_T(\pt)$ is the fraction field of $K_T(\pt)$. 

For any $\gamma\in K_T(\Fl_G)$, we have the localization at the fixed point $gI$ defined by
\[\gamma|_{g}:=\langle [\mathcal{O}_g],\gamma\rangle\in K_T(\pt). \]
Localizing both sides of the equations in \cref{thm:SMCrecaff} at the torus fixed point $gI$, we get the following result. 
\begin{cor}\label{cor:recSMCaff}
    For any simple root $\alpha_i\in I_\aff$, and $g,f\in W_\ext$, we have 
    \begin{align*}
        &(1+ye^{-\alpha_i})\SMC_y(\mathring{\Sigma}^{f})|_{s_ig}\\
    =&\begin{cases}
           -y(1-e^{-\alpha_i})\cdot s_i\bigg(\SMC_y(\mathring{\Sigma}^{s_if})|_g\bigg)+ (1+y)e^{-\alpha_i}\cdot s_i\bigg(\SMC_y(\mathring{\Sigma}^{f})|_{g}\bigg)&\textit{ if } s_if>f;\\
           (1-e^{-\alpha_i})\cdot s_i\bigg(\SMC_y(\mathring{\Sigma}^{s_if})|_g\bigg)+ (1+y) \cdot s_i\bigg(\SMC_y(\mathring{\Sigma}^{f})|_{g}\bigg) &\textit{ if } s_if<f. 
       \end{cases}
       \end{align*}
       and 
       \begin{align*}
           &(1+ye^{-g\alpha_i})\SMC_y(\mathring{\Sigma}^{f})|_{gs_i}\\
           =&\begin{cases}
          (1+y)\SMC_y(\mathring{\Sigma}^{f})|_g +(1-e^{-g\alpha_i})\SMC_y(\mathring{\Sigma}^{fs_i})|_g  & \textit{ if } fs_i<f;\\
         (1+y)e^{-g\alpha_i}\SMC_y(\mathring{\Sigma}^{f})|_g - y(1-e^{-g\alpha_i})\SMC_y(\mathring{\Sigma}^{fs_i})|_g  & \textit{ if } fs_i>f.\end{cases}
       \end{align*}
\end{cor}
\begin{proof}
    This follows from the explicit formula in \cref{equ:lDL,equ:TiTivee}, \cref{thm:SMCrecaff}, and the fact
    \[w^L(\gamma)|_u=w(\gamma|_{w^{-1}u}), ~w^R(\gamma)|_u=\gamma|_{uw}, ~\gamma\in K_T(\Fl_G).\]
\end{proof}
 
 Since we are considering the small torus $T$, $\delta$ is zero in $K_T(\pt)$. Hence, by \cref{equ:affact}, $e^{t_\mu(\alpha_i)}=e^{\alpha_i}\in K_T(\pt)$ for any $\mu\in X_*(T)$ and any simple root $\alpha_i$. Therefore, letting $g=t_{s_i\mu}$ in the first equation and $g=t_\mu$ in the second equation in \cref{cor:recSMCaff}, we get
\begin{cor}\label{cor:recSMCaffII}
    For any simple root $\alpha_i\in I_\aff$, $\mu\in X_*(T)$, and $f\in W_\ext$,
    \begin{enumerate}
        \item if $s_if>f$ and $fs_i>f$,
        \begin{align*}
            &-y(1-e^{-\alpha_i})\cdot s_i\bigg(\SMC_y(\mathring{\Sigma}^{s_if})|_{t_{s_i\mu}}\bigg)+ (1+y)e^{-\alpha_i}\cdot s_i\bigg(\SMC_y(\mathring{\Sigma}^{f})|_{t_{s_i\mu}}\bigg)\\=&(1+y)e^{-\alpha_i}\cdot \SMC_y(\mathring{\Sigma}^{f})|_{t_\mu} - y(1-e^{-\alpha_i})\cdot \SMC_y(\mathring{\Sigma}^{fs_i})|_{t_\mu}.
        \end{align*}
        \item if $s_if>f$ and $fs_i<f$, 
        \begin{align*}
            &-y(1-e^{-\alpha_i})\cdot s_i\bigg(\SMC_y(\mathring{\Sigma}^{s_if})|_{t_{s_i\mu}}\bigg)+(1+y)e^{-\alpha_i}\cdot s_i\bigg(\SMC_y(\mathring{\Sigma}^{f})|_{t_{s_i\mu}}\bigg)\\=&(1+y)\cdot \SMC_y(\mathring{\Sigma}^{f})|_{t_\mu} +(1-e^{-\alpha_i})\cdot \SMC_y(\mathring{\Sigma}^{fs_i})|_{t_\mu}.
        \end{align*}
        \item if $s_if<f$ and $fs_i>f$, 
        \begin{align*}
            &(1-e^{-\alpha_i})\cdot s_i\bigg(\SMC_y(\mathring{\Sigma}^{s_if})|_{t_{s_i\mu}}\bigg)+ (1+y) \cdot s_i\bigg(\SMC_y(\mathring{\Sigma}^{f})|_{t_{s_i\mu}}\bigg)\\
            =&(1+y)e^{-\alpha_i}\cdot \SMC_y(\mathring{\Sigma}^{f})|_{t_\mu} - y(1-e^{-\alpha_i})\cdot \SMC_y(\mathring{\Sigma}^{fs_i})|_{t_\mu}.
        \end{align*}
        \item if $s_if<f$ and $fs_i<f$, 
        \begin{align*}
            &(1-e^{-\alpha_i})\cdot s_i\bigg(\SMC_y(\mathring{\Sigma}^{s_if})|_{t_{s_i\mu}}\bigg)+ (1+y) \cdot s_i\bigg(\SMC_y(\mathring{\Sigma}^{f})|_{t_{s_i\mu}}\bigg)\\
            =&(1+y)\cdot \SMC_y(\mathring{\Sigma}^{f})|_{t_\mu} +(1-e^{-\alpha_i})\cdot \SMC_y(\mathring{\Sigma}^{fs_i})|_{t_\mu}. 
        \end{align*}
    \end{enumerate}
\end{cor}

\section{Comparison between the finite and affine case}
In this section, we prove our main result of this paper by relating the Segre motivic Chern classes of the open projected Richardson varieties in the finite partial flag variety $G/P$ and the Segre motivic Chern class of the Schubert cells in the affine flag variety $\Fl_G$. 

\subsection{Affine Grassmannian}
To build a bridge between $G/P$ and $\Fl_G$, we need to use the affine Grassmannian. 
Let $\Gr_G:= G(\!(z)\!)/G[[z]]$ be the {\it affine Grassmannian}. There is a natural projection map $\Fl_G\rightarrow
\Gr_G$ whose fibers are isomorphic to $G/B$. The $T$-fixed points $(\Gr_G)^T$ are in bijection with the coset $W_{\ext}/W\simeq X_*(T)$. For any $\lambda\in X_*(T)$, we let $t_\lambda W$ represent the fixed point $z^{-\lambda}G[[z]]/G[[z]]$. 
For any $\mu\in X_*(T)$ and $\gamma\in K_T(\Gr_G)$, let $\gamma|_{t_\mu W}$ be the localization at the fixed point $t_\mu W$ in $\Gr_G$.

Let $K\subset G$ be the maximal compact subgroup and $T_\mathbb{R}:=K\cap T$ the compact torus. Then we have
\[r\colon \Gr_G\simeq \Omega K\rightarrow LK\rightarrow LK/T_\mathbb{R}\simeq \Fl_G,\] 
where $LK$ (resp., $\Omega K$) is the space of polynomial maps $S^1\to K$ (resp. $(S^1,1)\to (K,1)$). Hence, there is a pullback map, which is called the wrong way map,
\[r^*: K_T(\Fl_G)\rightarrow K_T(\Gr_G).\]
In terms of localization, this is defined as follows:
\begin{equation}\label{equ:rloc}
    r^*(\gamma)|_{t_{\mu}W}=\gamma|_{t_\mu},
\end{equation}
where $\mu\in X_*(T)$, see \cite{LSS10}. Since $\Gr_G$ can be realized as a partial flag variety for the affine Kac-Moody group \cite{Ku02}, we have the operators $w^L$ on $K_T(\Gr_G)$ for any $w\in W_\ext$, and it satisfies the following formula:
\[w^L(\gamma)|_{t_\mu W}=w(\gamma|_{w^{-1}t_\mu W})=w(\gamma|_{t_{w^{-1}\mu}W}),\]
where $\gamma\in K_T(\Gr_G)$ and $\mu\in X_*(T)$. It is easy to check that $r^*$ does not commute with the operators $w^L$. But we have the following relation.
\begin{lem}\label{lem:r*wL}
    For any $\gamma\in K_T(\Fl_G)$, $\mu\in X_*(T)$, and any simple root $\alpha_i$ with $i\in I_\aff$,
    \[s_i(\gamma|_{t_{s_i\mu}})=\bigg(s_i^L(r^*\gamma)\bigg)|_{t_\mu W}.\]
\end{lem}
With \cref{equ:rloc,lem:r*wL}, \cref{cor:recSMCaffII} yields the following result. 
\begin{thm}\label{thm:recSMCGr}
    For any simple root $\alpha_i\in I_\aff$, $\mu\in X_*(T)$, and $f\in W_\ext$,
    \begin{enumerate}
        \item if $s_if>f$ and $fs_i>f$,
        \begin{align*}
            &-y(1-e^{-\alpha_i})\cdot s_i^L\bigg(r^*\SMC_y(\mathring{\Sigma}^{s_if})\bigg)+ (1+y)e^{-\alpha_i}\cdot s_i^L\bigg(r^*\SMC_y(\mathring{\Sigma}^{f})\bigg)\\=&(1+y)e^{-\alpha_i}\cdot r^*\SMC_y(\mathring{\Sigma}^{f}) - y(1-e^{-\alpha_i})\cdot r^*\SMC_y(\mathring{\Sigma}^{fs_i}).
        \end{align*}
        \item if $s_if>f$ and $fs_i<f$, 
        \begin{align*}
            &-y(1-e^{-\alpha_i})\cdot s_i^L\bigg(r^*\SMC_y(\mathring{\Sigma}^{s_if})\bigg)+(1+y)e^{-\alpha_i}\cdot s_i^L\bigg(r^*\SMC_y(\mathring{\Sigma}^{f})\bigg)\\=&(1+y)\cdot r^*\SMC_y(\mathring{\Sigma}^{f}) +(1-e^{-\alpha_i})\cdot r^*\SMC_y(\mathring{\Sigma}^{fs_i}).
        \end{align*}
        \item if $s_if<f$ and $fs_i>f$, 
        \begin{align*}
            &(1-e^{-\alpha_i})\cdot s_i^L\bigg(r^*\SMC_y(\mathring{\Sigma}^{s_if})\bigg)+ (1+y) \cdot s_i^L\bigg(\SMC_y(\mathring{\Sigma}^{f})\bigg)\\=&(1+y)e^{-\alpha_i}\cdot r^*\SMC_y(\mathring{\Sigma}^{f}) - y(1-e^{-\alpha_i})\cdot r^*\SMC_y(\mathring{\Sigma}^{fs_i}).
        \end{align*}
        \item if $s_if<f$ and $fs_i<f$, 
        \begin{align*}
            &(1-e^{-\alpha_i})\cdot s_i^L\bigg(r^*\SMC_y(\mathring{\Sigma}^{s_if})\bigg)+ (1+y) \cdot s_i^L\bigg(r^*\SMC_y(\mathring{\Sigma}^{f})\bigg)\\=&(1+y)\cdot r^*\SMC_y(\mathring{\Sigma}^{f}) +(1-e^{-\alpha_i})\cdot r^*\SMC_y(\mathring{\Sigma}^{fs_i}). 
        \end{align*}
    \end{enumerate}
\end{thm}

\subsection{Main result}
For any dominant cocharacter $\mu$, let $\Gr_\mu:=G[[z]]z^{-\mu} G[[z]]/G[[z]]\subset \Gr_G$ be the spherical Schubert cell. Then 
\[\Gr_G=\bigsqcup_{\mu\in X_*(T)^+}\Gr_\mu,\]
where $X_*(T)^+$ denotes the dominant cocharacters.
Let $\ev:G[[z]]\rightarrow G$ be the evaluation at $z=0$ map. 

Now let us fix a dominant cocharacter $\lambda$. Then we have a subset of the simple roots $\{\alpha_i\mid \langle \lambda,\alpha_i^\vee\rangle=0\}$, and let $P$ be the corresponding parabolic subgroup containing the Borel subgroup $B$. By \eqref{equ:affact},
\[\ev\bigg(G[[z]]\cap z^{-\lambda} G[[z]]z^\lambda\bigg)=P\subset G.\]
Hence,
\[\Gr_\lambda\simeq G[[z]]/(G[[z]]\cap z^{-\lambda} G[[z]]z^\lambda)\]
maps to $G/P$ via the map $\ev$. Moreover, it
is  an affine bundle over $G/P$, 
and denote the closed embedding  
\[i_\lambda:G/P\simeq G\cdot z^{-\lambda} G[[z]]/G[[z]]\subset \Gr_\lambda,\]
by regarding $G$ as a subgroup of $G[[z]]$ of constant loops, see \cite{Zhu}. We need the following easy lemma, see \cite{FGSX} for a proof.
\begin{lem}\label{lem:tanweGr}
    The $T\times\mathbb{C}^*_{\rot}$ weights of the tangent space $T_{t_{\lambda}W}\Gr_\lambda$ is
    \[\{-\alpha+k\delta\mid 0\leq k< \langle\lambda,\alpha\rangle\}.\]
\end{lem}

Consider the following diagram,
$$\xymatrix{
G/P\ar[r]^-{i_\lambda} & \Gr_\lambda\ar[r]^-{j_\lambda}& \Gr_G \ar[r]^-r &  \Fl_G,}$$
where $i_\lambda$ and $j_\lambda$ are inclusions, and let
\[q^*:=j_\lambda^*\circ r^*:K_T(\Fl_G)\rightarrow K_T(\Gr_\lambda).\]
Recall from \cref{def:OpenRic} that  for $u\in W$ and $w\in W^P$, we have the Segre motivic Chern class of the open projected Richardson variety $\SMC_y(\mathring{\Pi}_{f}), f=f^{\lambda}_{u,w}=ut_\lambda w^{-1}\in W_\ext$.  It is related to the SMC class $\SMC_y(\mathring{\Sigma}^{f})$ in \cref{def:SMCaff} as follows. 
\begin{thm}\label{thm:main}
    With the above notations and let $\mathcal{N}$ be the normal bundle of $G/P$ inside $\Gr_\lambda$, then
    \[i_{\lambda *}\bigg(\frac{\SMC_y(\mathring{\Pi}_{f})}{\lambda_y(\mathcal{N}^*)}\bigg)=q^*\bigg(\SMC_y(\mathring{\Sigma}^{f})\bigg)\in K_T(\Gr_\lambda)[[y]].\]
\end{thm}
\begin{rem}\label{rem:pushpull}
    The spherical Schubert cell $\Gr_\lambda$ has a $G$ action by left multiplication. Hence, $K_T(\Gr_\lambda)$ has a $W$ action as before. Since the map $i_{\lambda}$ is $G$-equivariant, the pushforward $i_{\lambda *}$ commutes with the action $w^L$. Besides, the normal bundle $\mathcal{N}$ is invariant under the $w^L$ action, as it has a $G$-equivariant structure. Hence, the classes $i_{\lambda *}\bigg(\frac{\SMC_y(\mathring{\Pi}_{f})}{\lambda_y(\mathcal{N}^*)}\bigg)$ satisfy the same recursive formulae as $\SMC_y(\mathring{\Pi}_{f})$ in \cref{thm:MCPif}. By the same reasoning, $q^*\bigg(\SMC_y(\mathring{\Sigma}^{f})\bigg)$ satisfies the same recursive formulae as $\SMC_y(\mathring{\Sigma}^{f})$ in \cref{thm:recSMCGr}.
\end{rem}
\begin{rem}
    Since the left hand side vanishes if $u\nleq w$, we get same vanishing result for the right hand side, which can also be obtained directly from localization.  
\end{rem}
\begin{proof} 
    We prove by induction on $\ell(w), w\in W^P$. Assume $w=\id$. Then the left hand side is $0$ if $u\neq \id$. On the other, all the torus fixed points in $\Gr_\lambda$ are of the form $t_{v\lambda}W$ for some $v\in W^\lambda$. If
    \[q^*\bigg(\SMC_y(\mathring{\Sigma}^{ut_\lambda})\bigg)|_{t_{v\lambda}W}=\SMC_y(\mathring{\Sigma}^{ut_\lambda})|_{t_{v\lambda}}\neq 0,\]
    then $ut_\lambda\leq t_{v\lambda}$, which implies $u=v=\id$ because of $\cref{equ:leng}$ and the fact $\ell(t_{v\lambda})=\ell(t_{\lambda})$. So in case $w=\id$, we can assume $u=\id$, that is,  $f=t_\lambda$, and it suffices to compare their restrictions at $t_{\la}W$.
    
    We first compute the LHS. $X(w)^\circ=\{e\cdot B\}\subset G/B$ is the identity point, where $e\in G$ is the unit element. Hence, $\MC_y(X(w)^\circ)=[\calO_{e\cdot B}]$ is the structure sheaf of the point. By \cref{def:MCR},
    \[\MC_y(\mathring{R}_{u,w})=[\calO_{e\cdot B}]\otimes \SMC_y(Y(u)^\circ)=\frac{\MC_y(Y(u)^\circ)|_{e\cdot B}}{\lambda_y(T^*(G/B))|_{e\cdot B}}[\calO_{e\cdot B}]=[\calO_{e\cdot B}],\]
    where the last equality follows from the fact that $Y(u)^\circ$ is smooth at the point $e\cdot B$, and the closure of $Y(u)^\circ$ is the whole flag variety $G/B$, see \cite{AMSS24}. Therefore, by \cref{def:OpenRic},
    \[\SMC_y(\mathring{\Pi}_{f})=\frac{[\calO_{e\cdot P}]}{\lambda_y(T^*(G/P))}=\frac{[\calO_{e\cdot P}]}{\lambda_y(T^*(G/P))|_{e\cdot P}},\]
    where $e\cdot P\in G/P$ is the identity point in $G/P$. Hence,
    \begin{equation}\label{equ:lhs}
        i_{\lambda *}\bigg(\frac{\SMC_y(\mathring{\Pi}_{f})}{\lambda_y(\mathcal{N}^*)}\bigg)=\frac{[\mathcal{O}_{t_\lambda W}]}{\lambda_y(T^*\Gr_\lambda)|_{t_\lambda W}}=\frac{[\mathcal{O}_{t_\lambda W}]}{\prod_{\alpha,\langle\lambda,\alpha\rangle>0}(1+ye^{\alpha})^{\langle\lambda,\alpha\rangle}}.
    \end{equation}
    Here $\mathcal{O}_{t_\lambda W}$ is the structure sheaf of the fixed point $t_\lambda W\in \Gr_G$, and the second equality follows from \cref{lem:tanweGr}. 
    
    Let us now compute the RHS with $u=w=\id$. First of all, for any $g\in W_\ext$, the motivic Chern class $\MC_y(\mathring{\Sigma}_{g})$ is supported on the Schubert variety $\overline{\mathring{\Sigma}_{g}}$. Hence, the dual basis $\widetilde{\MC}_y(\mathring{\Sigma}^{g})$ defined in \cref{equ:defdual} has zero localization at the fixed points $h\in W_\ext$ if $g\nleq h$. Therefore, by \cref{def:SMCaff}, for any $v\neq id\in W$,
    \[q^*\bigg(\SMC_y(\mathring{\Sigma}^{f})\bigg)|_{t_{v\lambda}W}=0.\]
    Notice that all the torus fixed points in $\Gr_\lambda$ are of the form $t_{v\lambda}W$ for some $v\in W$. Therefore, $q^*\bigg(\SMC_y(\mathring{\Sigma}^{f})\bigg)$ is a multiple of the structure sheaf of the fixed point $[\mathcal{O}_{t_\lambda W}]$. Let us now compute this scalar.
    Since the Schubert cell $\mathring{\Sigma}_{t_\lambda}$ is smooth at the torus fixed point $t_\lambda\in \Fl_G$, we get 
    \[\MC_y(\mathring{\Sigma}_{t_\lambda})=\sum_{g\leq t_\lambda} a_g[e_g]\]
    for some coefficients $a_g\in \Frac K_T(\pt)$ with 
    \[a_{t_\lambda}=\prod_\chi \frac{1+ye^{-\chi}}{1-e^{-\chi}}=\prod_{\alpha,\langle\lambda,\al\rangle>0}\bigg(\frac{1+ye^\alpha}{1-e^\al}\bigg)^{\langle\lambda,\al\rangle}.\]
    Here the $\chi$ in the middle is the $T$-weights of the tangent space $T_{t_\lambda}\mathring{\Sigma}_{t_\lambda}$, the first equality follows\cite[Lemma 9.1]{AMSS24}, and second equality follows from \cref{lem:tangentwt}. Hence,
    \[[e_{t_\lambda}]=\sum_{g\leq t_\lambda} b_g \MC_y(\mathring{\Sigma}_g)\]
    for some coefficients $b_g\in \Frac K_T(\pt)$, with 
    \[b_{t_\lambda}=\frac{1}{a_{t_\lambda}}=\prod_{\alpha,\langle\lambda,\al\rangle>0}\bigg(\frac{1-e^\al}{1+ye^\alpha}\bigg)^{\langle\lambda,\al\rangle}.\]
    Hence,
    \begin{align*}
        q^*\bigg(\SMC_y(\mathring{\Sigma}^{f})\bigg)|_{t_{\lambda}W}=&\widetilde{\MC}(\mathring{\Sigma}^{t_\lambda})|_{t_\lambda}=\langle [e_{t_\lambda}],\widetilde{\MC}(\mathring{\Sigma}^{t_\lambda})\rangle\\
        =&b_{t_\lambda}=\prod_{\alpha,\langle\lambda,\al\rangle>0}\bigg(\frac{1-e^\al}{1+ye^\alpha}\bigg)^{\langle\lambda,\al\rangle}\\
        =&i_{\lambda *}\bigg(\frac{\SMC_y(\mathring{\Pi}_{f})}{\lambda_y(\mathcal{N}^*)}\bigg)|_{t_{\lambda}W}.
    \end{align*}
    where the last equality follows from \cref{equ:lhs}. This concludes the proof for the base case $w=\id$.

    Now pick a $w\in W^P$ such that $w\neq \id$, and suppose 
    \begin{equation}\label{equ:assump}
        i_{\lambda *}\bigg(\frac{\SMC_y(\mathring{\Pi}_{f^\lambda_{u,v}})}{\lambda_y(\mathcal{N}^*)}\bigg)=q^*\bigg(\SMC_y(\mathring{\Sigma}^{f^\lambda_{u,v}})\bigg)
    \end{equation}
    for all $u\in W$, $v\in W^P$, such that $\ell(v)<\ell(w)$. There exists a simple root $\alpha_i$ so that $s_iw<w$. In particular, $v:=s_iw\in W^P$ by \cref{lem:parabolic}. Letting $f=f^\lambda_{u,v}$ in \cref{thm:MCPif} and \cref{thm:recSMCGr}. Since $s_iw<w$, we get $fs_i=f^\lambda_{u,w}<f$ by \cref{equ:leng}. If $s_if<f$, by combining the recursions in \cref{thm:MCPif}.(1), \cref{thm:recSMCGr}.(4), \cref{rem:pushpull}, and the induction assumption \cref{equ:assump}, we get the Theorem for $fs_i=f^\lambda_{u,w}$. If $s_if>f$, we use the recursions in \cref{thm:MCPif}.(3) and \cref{thm:recSMCGr}.(2) to get the same conclusion. This finishes the proof of the theorem.
\end{proof}

\section{Localization of SMC classes and Twisted R-polynomials}\label{sec:twistedR}

In this section, we study the relation between the localization of the SMC classes $\SMC_y(Y(u)^\circ)$ and the twisted R-polynomials from \cite{GLTW},
which is a generalization of the Kazhdan--Lusztig R-polynomials in \cref{sec:KLRpoly}. 
%We first derive certain subword combinatorics. 

\def\ww{\mathbf{w}}
\def\uu{\mathbf{u}}

\subsection{Subwords}
 % Many of definitions are from \cite{GLTW}. 
Let $u,w\in W$. A $w$-word is a sequence  $\ww=(s_{i_1},\ldots,s_{i_{\ell}})$ of simple reflections such that $w=s_{i_1}\cdots s_{i_\ell}$. 
The word $\ww$ is said to be \emph{reduced} if $\ell = \ell(w)$. 
We define 
\begin{equation*}
\beta_{\ww,k} = s_{i_1}\cdots s_{i_{k-1}}\alpha_{i_k},\qquad k=1,2\ldots,\ell. 
\end{equation*}
If $\ww$ is reduced, 
$$\{\beta_{\ww,1},\ldots,\beta_{\ww,k}\}
=\{\alpha>0:w^{-1}\alpha<0\}.$$
We denote $\be_k=\be_{\ww,k}$ for simplicity. A $u$-subword of $\ww$ is a sequence 
$\uu=(\uu_1,\ldots,\uu_\ell)$ such that $\uu_k\in \{s_{i_k},e\}$ for $k=1,\ldots,\ell$ and 
$u=\uu_1\cdots \uu_\ell$. 
We set the partial product 
% $\mathbf{w}_{(k)}=s_{i_1}\cdots s_{i_k}$ and 
$\mathbf{u}_{(k)}=\uu_{1}\cdots \uu_{k}$, and  introduce subsets $J_{\uu}^{\pm},E_{\uu}^{\pm}\subset [\ell]:=\{1,2,\cdots,\ell\}$ by the following table:
$$
\begin{array}{c|c|c}\hline
\vphantom{\dfrac12}
\text{Cases }& \uu_k=s_{i_k} & \uu_k=e\\\hline
\vphantom{\dfrac12}
\uu_{(k-1)} s_{i_k} < \uu_{(k-1)}& 
k\in J_{\uu}^- & 
k\in E_{\uu}^-\\\hline
\vphantom{\dfrac12}
\uu_{(k-1)} s_{i_k} > \uu_{(k-1)}& 
k \in J_{\uu}^+ & 
k\in E_{\uu}^+\\\hline
\end{array}\qquad 
\begin{aligned}
J_{\uu} & = J_{\uu}^+\sqcup J_{\uu}^-,\\
E_{\uu} & = E_{\uu}^+\sqcup E_{\uu}^-,\\
[\ell]& =J_{\uu}\sqcup E_{\uu}. 
\end{aligned}$$
% For readers familiar with \cite{}, 
Then $\uu$ is called \emph{reduced} if $\uu_{(k-1)}\leq \uu_{(k)}$ for all $k$, i.e. 
$J_{\uu}^-=\varnothing$. Similarly, 
$\uu$ is called \emph{distinguished} if $\uu_{(k)}\leq \uu_{(k-1)}s_{i_k}$ for all $k$, i.e. 
$E_{\uu}^-=\varnothing$. 

\begin{example}

Consider the case of $A_4$ with 
$$\ww = (s_4,s_3,s_1,s_4,s_2,s_1,s_3,s_2),\qquad 
u = s_3s_4s_3s_2.$$
There are five $u$-subwords of $\ww$, listed in the following table (here  $\c{i}$ stands for $e\in W$):
$$\begin{array}{ccccccc}\uu&J_\uu^+&J_\uu^-&E_\uu^+&E_\uu^-&\text{reduced}&\text{distinguished}\\\hline
    (\m4,\m3,\c1,\m4,\m2,\c1,\c3,\c2)&\{1,2,4,5\}&\varnothing&\{3,6,7\}&\{8\}&\text{yes} & \text{no}\\\hline
    (\m4,\m3,\c1,\m4,\c2,\c1,\c3,\m2)&\{1,2,4,8\}&\varnothing&\{3,5,6\}&\{7\}&\text{yes}&\text {no}\\\hline
    (\c4,\m3,\c1,\m4,\c2,\c1,\m3,\m2)&\{2,4,7,8\}&\varnothing&\{1,3,5,6\}&\varnothing&\text{yes}& \text{yes}\\\hline
    (\m4,\m3,\m1,\m4,\c2,\m1,\c3,\m2)&\{1,2,3,4,8\}&\{6\}&\{5\}&\{7\}&\text{no}& \text{no}\\\hline
    (\c4,\m3,\m1,\m4,\c2,\m1,\m3,\m2)&\{2,3,4,7,8\}&\{6\}&\{1,5\}&\varnothing&\text{no} & \text{yes}\\\hline
\end{array}$$
The $R$-polynomial is 
$$R_{u,w} 
= (q-1)^4+q(1-q)^2
= 1 - 3q + 4q^2 - 3q^3 + q^4.$$

\iffalse
$$\begin{array}{cccccc}\hline
\vphantom{\dfrac12}\uu&J_\uu^+&J_\uu^-&E_\uu^+&E_\uu^-&\text{properties}\\\hline
\vphantom{\dfrac12}
    (\m4,\m3,\c1,\m4,\m2,\c1,\c3,\c2)&\{1,2,4,5\}&\varnothing&\{3,6,7\}&\{8\}&\text{reduced, not distinguished}\\\hline
\vphantom{\dfrac12}
    (\m4,\m3,\c1,\m4,\c2,\c1,\c3,\m2)&\{1,2,4,8\}&\varnothing&\{3,5,6\}&\{7\}&\text{reduced, not distinguished}\\\hline
\vphantom{\dfrac12}
    (\c4,\m3,\c1,\m4,\c2,\c1,\m3,\m2)&\{2,4,7,8\}&\varnothing&\{1,3,5,6\}&\varnothing&\text{reduced and distinguished}\\\hline
\vphantom{\dfrac12}
    (\m4,\m3,\m1,\m4,\c2,\m1,\c3,\m2)&\{1,2,3,4,8\}&\{6\}&\{5\}&\{7\}&\text{not reduced, not distinguished}\\\hline
\vphantom{\dfrac12}
    (\c4,\m3,\m1,\m4,\c2,\m1,\m3,\m2)&\{2,3,4,7,8\}&\{6\}&\{1,5\}&\varnothing&\text{not reduced,  distinguished }\\\hline
\end{array}$$
\fi

\end{example}

%\begin{example}Consider the case of $A_2$ with $w=s_2s_1$, $\ww=(s_1,s_1,s_2,s_1)$, then 
%\[J_{(e,s_1,s_2,s_1)}^+=\{2,3,4\}, \quad E_{(e,s_1,s_2,s_1)}^+=\{1\},\]so $(e,s_1,s_2,s_1)$ is reduced and distinguished. 
%We have
%\[J_{(s_1,s_1,s_2,e)}^+=\{1,3\},  \quad E_{(s_1,s_1,s_2,e)}^+=\{4\},\]
%so $(s_1,s_1,s_2,e)$ is not reduced but is distinguished. 
%Similarly, 
%\[
%J_{(s_1,e,s_2,s_1)}^+=\{1,3,4\}, \quad E_{(s_1,e,s_2,s_1)}^-=\{2\},
%\]
%so it is reduced but not distinguished. 

%Lastly, 
%\[
%J^+_{s_1,e, e, s_1}=\{1\}, \quad E^+_{s_1,e, e, s_1}=\{3\},
%\]
%so it is not reduced nor distinguished. 
%\end{example}
\iffalse
\begin{example}
{\color{red} do we need this example?}
We  consider the case of $A_2$ with $w=s_1s_2s_1, \ww=\{s_1,s_2,s_1\}$.  
% \[
% J^+_{(s_1,e,e)}=\{1\},\quad E^-_{(s_1, e,e,)}=\{3\}, \quad E^+_{(s_1,e,e)}=\{2\}.
% \]
% \[
% J^+_{(e,e,s_1)}=\{3\}, \quad  E^+_{(e,e,s_1)}=\{1,2\}.
% \]
% so it is reduced but not distinguished. 
There are two $e$-subwords:
\[\begin{array}{ccccccc}
\uu& J_\uu^+& J^-_\uu & E^+_\uu&E^-_\uu & \text{reduced}&\text{distinguished}\\
\hline
(\c1,\c2,\c1)& \varnothing & \varnothing & \{1,2,3\} & \varnothing & \text{yes}& \text{yes}\\
(\m1,\c2,\m1)& \{1\} & \{3\}& \{2\} & \varnothing & \text{no} & \text{yes}
\end{array}
\]
And it is easy to verify  that \[R_{e,s_1s_2s_1}=(q-1)^3+q(q-1).\] 
We are going to see that $R$-polynomials are related to distinguished subwords. 
\end{example}
\fi
 
\subsection{Subword recursion} 
Let $\ww$ be a $w$-word and $\ww \sqcup s$ be  a $ws$-word with $s$ a simple reflection.  For any $u$, there are two types of $u$-subwords in $\ww s$: 
\begin{itemize}
\item Type I: $\xx\sqcup \c{}$ with $\xx$ a $u$-subword of $\ww$;
 \item Type II: $\xx'\sqcup s$ with $\xx'$ a $us$-subword of $\ww$. 
 \end{itemize}
 Moreover, Type I subwords are in bijection to $u$-subwords $\xx$ of $\ww$, and Type II subwords are in bijection to $us$-subwords $\xx'$ of $\ww$. 
 By definition, if we compare the $J^\pm$-subsets and $E^\pm$-subsets  of $\xx\sqcup \c{}$ and  $\xx$ (or  $\xx'\sqcup s$  and $\xx'$, respectively), they will coincide,  except for the following cases
 \begin{equation}\label{eq:typemat}
 \begin{array}{c|c|c}
 &\text{Type I} & \text{Type II}\\
 \hline
 us<u&E^-_{\xx\sqcup \c{}}=E^-_\xx\sqcup \{s\} & J^+_{\xx'\sqcup s}=J^+_{\xx'}\sqcup \{s\}\\
 us>u&E^+_{\xx\sqcup \c{}}=E^+_{\xx}\sqcup \{s\}&J^-_{\xx'\sqcup s}=J^-_{\xx'}\sqcup \{s\}
\end{array}.
 \end{equation}
 Here we abuse the notation and refer the last term in the word $\ww \sqcup s$ as $s$. 

Assume there is a collection of elements $\calS_{u,w}, v,w\in W$  with $\calS_{u,\id}=\delta_{u,\id}$ and subject to the recursive relation ($s=s_\al$):
 \begin{equation}\label{eq:Srec}
\calS_{u,ws}=\begin{cases}\rmp_{11}(w\al)\calS_{u,w}+\rmp_{12}(w\al)\calS_{us,w}, & us<u;\\
 \rmp_{21}(w\al)\calS_{u,w}+\rmp_{22}(w\al)\calS_{us,w}, & us>u.
 \end{cases}
 \end{equation}
 Here $p_{ij}(\al)$ for any root $\al$ are certain elements (and   could be constants that do not depend on $\al$).  We refer to \eqref{eq:Srec} as `subword recursion' condition.  Examples of $\calS_{u,w}$ include various cohomology and K-theory classes associated to Schubert cells/varieties (\eqref{eq:Sch} and \eqref{eq:SMC}), and also the $R$-polynomials \eqref{eq:Rrec}. 
 
 It turns out that the subword recursion condition can be  transformed into a sub-word formula. For each $w$-word $\ww$, with $\be_k=\be_{\ww,k}$,  we have 
 \begin{equation}\label{eq:Sword}
 \calS_{u,w}=\sum_{u\text{-subwords } \uu}~~\prod_{k=1}^\ell \begin{cases}\rmp_{11}(\be_k), & k\in E^-_{\uu};\\
 \rmp_{12}(\be_k), & k\in J^+_\uu;\\
 \rmp_{21}(\be_k), & k\in E^+_{\uu};\\
 \rmp_{22}(\be_k), & k\in J^-_{\uu}. \end{cases}
 \end{equation}
Indeed, it suffices to show that the RHS satisfies the recursive relation as \eqref{eq:Srec}. This can be proved by using \eqref{eq:typemat}.
 
We recall two famous combinatorial formulas that can be derived from subword recursions.  The classical AJS-Billey (Anderson, Jantzen, Soergel) formula \cite{AJS94, Bi99} states the localization of the fundamental class of the opposite Schubert variety $[Y(u)]\in H_T^*(G/B)$ at the torus fixed point $\dot{w}B\in G/B$:
\begin{equation}\label{eq:billeycomb}
[Y(u)]\big|_{w} = \sum_{\uu}\prod_{k\in J_{\uu}}\beta_{\ww,k},
\end{equation}
with the sum over reduced $u$-subwords of $\ww$ (so $J^-_\uu=\varnothing$). 
Indeed, in this case it is well-known that if $ws_\al>w$, then 
\begin{equation}\label{eq:Sch}
[Y(u)]_{ws_\al}=\begin{cases}[Y(u)]|_w+w(\al)[Y(us_\al)]|_{w}, & us_\al<u;\\
[Y(u)]|_{w},& us_\al>u. 
\end{cases}
\end{equation}
One can then derive \eqref{eq:billeycomb} by using  \eqref{eq:Sword}. 

Corresponding formulas for the Segre--MacPherson classes (cohomological stable basis), K-theoretic Schubert classes, SMC classes, and elliptic Schubert classes can be found in \cite{G02, S17, SZZ,LXZ25}. 
    
    The second one is for $R$-polynomials (see Section \ref{sec:KLRpoly}). Assuming $ws>w$, one can rewrite the recursive formula in Section \ref{sec:KLRpoly} as 
\begin{equation}\label{eq:Rrec}
R_{u,ws}=\begin{cases}R_{us,w}, & us<u;\\
(q-1)R_{u,w}+qR_{us,w}, & us>u.
\end{cases}
\end{equation} Therefore, 
 if  $\ww$ be a reduced word of $w$, then  it follows from \eqref{eq:Sword} that
\begin{equation}\label{eq:Rcomb}
R_{u,w}(q)=\sum_{\uu} q^{|J_{\uu}^-|}(q-1)^{|E_{\uu}|},
\end{equation}
with the sum over distinguished $u$-subwords of $\ww$ (so $E^-_\uu=\varnothing$).

We derive a subword formula for the localization of $\SMC_y(Y(u)^\circ)$, which is slightly more explicit than \cite{SZZ}. 
In type $A$, a diagrammatic way of presenting it can be found in \cite[Proposition 2.1]{KJII}.
\begin{prop}\label{prop:billey}
For any $w$-word $\ww$ (not necessarily reduced), we have
$$\SMC_y(Y(u)^\circ)|_w = \sum_{\text{$u$-subwords $\uu$}}(-y)^{|J_{\uu}^-|}
\prod_{k=1}^{\ell}
\begin{cases}
\dfrac{e^{\beta_k}-1}{e^{\beta_k}+y} ,& k\in J_{\uu},\\[2ex]
\dfrac{(1+y)e^{\beta_k}}{e^{\beta_k}+y} ,& k\in E_{\uu}^-,\\[2ex]
\dfrac{1+y}{e^{\beta_k}+y} ,& k\in E_{\uu}^+,
\end{cases}.
$$
\end{prop}
\begin{proof}
First of all, if $w=\id$, $\SMC_y(Y(u)^\circ)|_w=0$ if $u\neq \id$. If $u=\id$, then $\SMC_y(Y(\id)^\circ)|_{\id}=\sum_v\SMC_y(Y(v)^\circ)|_{\id}=1$. Moreover, the right hand side also equals $\delta_{u,\id}$ if $w=\id$. Thus, both sides are equal if $w=\id$.
By the definition of $T^{R,\vee}_i$ in \eqref{equ:TiTivee}, we have 
$$\gamma|_{ws_i}
=
-\frac{1-e^{w\alpha_i}}{y+e^{w\alpha_i}}(T_i^{R,\vee}\gamma)|_{w}
+\frac{(1+y)e^{w\alpha_i}}{y+e^{w\alpha_i}}\gamma|_w
$$
for any class $\gamma$. 
By \eqref{eq:TRdual}, we get
\begin{equation}
\label{eq:SMC}
\SMC_y(Y(u)^\circ)|_{ws_i}
=\begin{cases}
\dfrac{e^{w\alpha_i}-1}{e^{w\alpha_i}+y} 
    \SMC(Y(us_i)^\circ)|_w
+
\dfrac{(1+y)e^{w\alpha_i}}{e^{w\alpha_i}+y}
    \SMC(Y(u)^\circ)|_w,
    &us_i<u,\\[2ex]
-y \dfrac{e^{w\alpha_i}-1}{e^{w\alpha_i}+y} 
    \SMC(Y(us_i)^\circ)|_w
+\dfrac{1+y}{e^{w\alpha_i}+y}
    \SMC(Y(u)^\circ)|_w,
    &us_i>u.
\end{cases}
\end{equation}
Now the formula follows immediately from \eqref{eq:Sword}. 
\end{proof}

It turns our that this formula specializes to the AJS-Billey formula. 
By \cite[Theorem 5.1]{AMSS24b}, the $y=0$ specialization of $\SMC_y(Y(u)^\circ)$ is 
\[\SMC_0(Y(u)^\circ)=\MC_0(Y(u)^\circ)=[\mathcal{O}_{Y(u)}(-\partial Y(u))]=\sum_{w\geq u}(-1)^{\ell(w)-\ell(u)}[\mathcal{O}_{Y(w)}],\] 
the ideal sheaf of the boundary of the Schubert variety $Y(u)$. Via the equivariant Chern character map $\ch$, we see that $[Y(u)]$ is the lowest degree term of $\ch(\SMC_0(Y(u)^\circ))$. Hence, if we expand $e^{\beta}=1+\beta+\frac{\beta^2}{2!}+\cdots$ for any root $\beta$, the lowest degree term in $\SMC_0(Y(u)^\circ)|_w$ is precisely $[Y(u)]|_w$. Since $y=0$, only the $u$-subwords with $J_{\uu}^-=\varnothing$, i.e. reduced subwords, will contribute to $\SMC_0(Y(u)^\circ)|_w$. 
Moreover, 
\begin{align*}
\lim_{y\to 0} 
\dfrac{e^{\beta_k}-1}{e^{\beta_k}+y}
& = 1-e^{-\beta_k} = \beta_k+\cdots,\\
\lim_{y\to 0} 
\dfrac{(1+y)e^{\beta_k}}{e^{\beta_k}+y}
& = 1,\qquad 
\lim_{y\to 0} 
\dfrac{1+y}{e^{\beta_k}+y}
 = e^{-\beta_k} = 1+\cdots 
\end{align*}
It is immediate to see that \eqref{eq:SMC} converges to \eqref{eq:Sch}, so the formula in Proposition \ref{prop:billey} converges to the AJS-Billey formula.

Now let us explain how this formula relates to the formula for $R$-polynomial \eqref{eq:Rcomb}. This limit process has already appeared in \cite[Theorem 1.2]{BN19} from a different perspective. 

\begin{prop}
\label{prop:limittoRpoly}
For $u,w\in W$, we have 
$$\lim_{e^{\alpha_i}\to 0}
\prod_{\alpha>0}\frac{1+y{e^{w\alpha}}}{1-e^{w\alpha}}
\SMC(Y(u)^\circ)|_w = R_{u,w}(-y), $$
where $\lim\limits_{e^{\alpha_i}\to 0}$ means we let $e^{\alpha_i}$ approach $0$ for any simple root.
When $\ww$ is reduced, each summand in Proposition \ref{prop:billey} converges to the corresponding summand in \eqref{eq:Rcomb}.  
\end{prop}
\begin{proof}
Under the assumption that $\ww$ is a reduced word, $\beta_1,\ldots,\beta_k$ are all positive roots. 
% $$\{\beta_1,\ldots,\beta_\ell\} = \{\alpha>0: w^{-1}\alpha<0\}. $$
The formula follows from the following computation
% \begin{equation}\label{eq:limtoR}
\begin{align*}
\lim_{e^{\alpha_i}\to 0}
\prod_{\alpha>0}\frac{1+y{e^{w\alpha}}}{1-e^{w\alpha}}& =
\prod_{\alpha>0}
\begin{cases}
1,    & \text{ if }w\alpha>0\\
(-y),  & \text{ if }w\alpha<0
\end{cases} =(-y)^{\ell(w)}. \\
\lim_{e^{\alpha_i}\to 0}
\frac{e^{\beta_k}-1}{e^{\beta_k}+y}
& = (-y)^{-1},~
\lim_{e^{\alpha_i}\to 0}
\dfrac{(1+y)e^{\beta_k}}{e^{\beta_k}+y}  =0,~
\lim_{e^{\alpha_i}\to 0}
\dfrac{(1+y)}{e^{\beta_k}+y} =(-y)^{-1}((-y)-1). \qedhere
\end{align*}
% \end{equation}
\end{proof}

\subsection{Twisted R-polynomials}
Let $v$ be another Weyl group element. 
Similar to the usual R-polynomials in Section \ref{sec:KLRpoly}, the twisted R-polynomial $R_{u,w}^{(v)}$ is defined by the following formula (see \cite[Proposition 5.1]{GLTW}):
\begin{equation}\label{equ:def}
    T_v T^{-1}_{w^{-1}} = q^{\ell(v)}\sum_{u\in W} R_{u,w}^{(v)}(q^{-1})q^{-\ell(vu)} T_{vu}\in H(W).
\end{equation}
There is also a combinatorial formula for it. We define subsets $J^{(v),\pm}_{\uu},E^{(v),\pm}_{\uu}\subset [\ell]$ by the following table:
$$
\begin{array}{c|c|c}\hline
\vphantom{\dfrac12}
\text{Cases }& \uu_k=s_{i_k} & \uu_k=e\\\hline
\vphantom{\dfrac12}
v\uu_{(k-1)} s_{i_k} < v\uu_{(k-1)}& 
k\in J^{(v),-}_{\uu} & 
k\in E^{(v),-}_{\uu}\\\hline
\vphantom{\dfrac12}
v\uu_{(k-1)} s_{i_k} > v\uu_{(k-1)}& 
k \in J^{(v),+}_{\uu} & 
k\in E^{(v),+}_{\uu}\\\hline
\end{array}\qquad
\begin{aligned}
J_{\uu}^{(v)} & = J_{\uu}^{(v),+}\sqcup J_{\uu}^{(v),-},\\
E_{\uu}^{(v)} & = E_{\uu}^{(v),+}\sqcup E_{\uu}^{(v),-},\\
[\ell]& =J^{(v)}_\uu\sqcup E^{(v)}_\uu. 
\end{aligned}$$
The twisted $R$-polynomial admits the following formula \cite[Section 4]{GLTW}: 
$$R^{(v)}_{u,w}(q)=\sum_{\uu} q^{|J_{\uu}^{(v),-}|}(q-1)^{|E_{\uu}^{(v)}|}$$
with the sum over $v$-distinguished $u$-subwords of $\ww$, i.e. $E_{\uu}^{(v),-}=\varnothing$. 

\begin{thm}For $u,v,w\in W$, we have 
$$\lim_{e^{v\alpha_i}\to 0}
\prod_{\alpha>0}\frac{1+y{e^{w\alpha}}}{1-e^{w\alpha}}
\SMC_y(Y(u)^\circ)|_w
= R_{v^{-1}u,v^{-1}w}^{(v)}(-y). 
$$
\end{thm}
\begin{rem} 
Let us explain the meaning of the limit in the theorem. For each $v\in W$, we can view 
$e^{v\alpha_i}$ for $i\in I$ as independent variables, 
and $e^\alpha$ for any root $\alpha$ is viewed as a function in these variables. 
Then $\lim\limits_{e^{v\alpha_i}\to 0}$ is the usual limit for multivariate functions in calculus. 
%  % $SMC(Y(u))|_w$ 
% Formally, we can view elements of $K_T(pt)_{loc}$  rational functions over the torus $T$. 
% Then for a function $f$ and $v\in W$, we define (if exists)
% $$\lim_{e^{v\alpha_i}\to 0} f = 
% \lim_{t\to \infty} f(\gamma(t))
% \text{ for any curve $\gamma$ with }
% \lim_{t\to \infty} e^{v\alpha_i}(\gamma(t))=0.$$
In fancy language, we can think of $f$ as a rational function over the toric variety associated with the Coxeter fan. The limit $\lim\limits_{e^{v\alpha_i}\to 0}f$ (if defined) is the value of $f$ at the torus fixed point corresponding to the Weyl chamber of $v$. 
\end{rem}
\begin{proof}
We prove the following equivalent form (via changing of variable $(w,u)\mapsto (vw,vu)$)
\begin{equation}\label{eq:limit}\lim_{e^{v\alpha_i}\to 0}
\prod_{\alpha>0}\frac{1+y{e^{vw\alpha}}}{1-e^{vw\alpha}}
\SMC_y(Y(vu)^\circ)|_{vw}
= R_{u,w}^{(v)}(-y). 
\end{equation}
\def\vv{\mathbf{v}}% 
\def\xx{\mathbf{x}}% 
Pick a reduced $v$-word $\vv = (s_{i_1'},\ldots,s_{i_h'})$ and a reduced $w$-word $\ww=(s_{i_1},\ldots,s_{i_\ell})$. 
Now the concatenation $\vv\ww$ is a $vw$-word. Denote 
$\gamma_k = \beta_{\vv,k}$ and 
$\beta_k = \beta_{\ww,k}$. 
Then we have 
\[\beta_{\vv\ww,k}
= \gamma_k\qquad (k=1,\ldots,h),\qquad 
\beta_{\vv\ww,h+k}
= v\beta_{k},\qquad 
(k=1,\ldots,\ell). \]
Proposition \ref{prop:billey} implies that  the LHS of \eqref{eq:limit} admits a formula in terms of $vu$-subwords $\xx$ of $\vv\ww$, i.e. 
$$\sum_{\text{$vu$-subword } \xx}
(-y)^{|J_{\xx}^-|}A_{\xx}B_{\xx}$$
where 
$$A_{\xx} = 
\prod_{k=1}^{h}
\begin{cases}
\dfrac{e^{\gamma_k}-1}{e^{\gamma_k}+y} & k\in J_{\xx},\\[2ex]
\dfrac{(1+y)e^{\gamma_k}}{e^{\gamma_k}+y} & k\in E_{\xx}^-,\\[2ex]
\dfrac{1+y}{e^{\gamma_k}+y} & k\in E_{\xx}^+.
\end{cases}\qquad 
B_{\xx} = 
\prod_{k=1}^{\ell}
\begin{cases}
\dfrac{e^{v\beta_k}-1}{e^{v\beta_k}+y} & h+k\in J_{\xx},\\[2ex]
\dfrac{(1+y)e^{v\beta_k}}{e^{v\beta_k}+y} & h+k\in E_{\xx}^-,\\[2ex]
\dfrac{1+y}{e^{v\beta_k}+y} & h+k\in E_{\xx}^+.
\end{cases}$$
First we claim that 
$$\lim_{e^{v\alpha_i}\to 0}A_{\xx}=\begin{cases}
1& \text{if }(\xx_1,\ldots,\xx_h) = \vv,\\
0& \text{otherwise}.
\end{cases}$$
Since $v^{-1}\gamma_k$ are negative roots, we get
\begin{align*}
\lim_{e^{v\alpha_i}\to 0 }
\frac{e^{\gamma_k}-1}{e^{\gamma_k}+y}
& = 
v\left(\lim_{e^{\alpha_i}\to 0 }
\frac{e^{v^{-1}\gamma_k}-1}{e^{v^{-1}\gamma_k}+y}\right) = 1,\\
\lim_{e^{v\alpha_i}\to 0 }
\dfrac{(1+y)e^{\gamma_k}}{e^{\gamma_k}+y}
& = v\left(\lim_{e^{\alpha_i}\to 0 }\dfrac{(1+y)e^{v^{-1}\gamma_k}}{e^{v^{-1}\gamma_k}+y}
\right) = 1+y,\\
\lim_{e^{v\alpha_i}\to 0 }
\dfrac{1+y}{e^{\gamma_k}+y}
& = v\left(\lim_{e^{\alpha_i}\to 0 }\dfrac{1+y}{e^{v^{-1}\gamma_k}+y}
\right) = 0.
\end{align*}
As a result, the limit of $A_{\xx}$ is nonzero only if $E_{\xx}^+\cap \{1,\ldots,h\}=\varnothing$, which is equivalent to $\xx_k=\vv_k$ for all $k=1,\ldots,h$ by induction on $k$. 
So we can assume $\mathbf{x} = \mathbf{v}\mathbf{u}$ for a $u$-subword of $\ww$. 
Since $\xx_{(h+k)}=v\uu_{(k)}$, we have 
$$
J_{\xx}^+ = \{1,\ldots,h\}\cup (h+J_{\uu}^{(v),+}),\qquad 
J_{\xx}^- = h+J_{\uu}^{(v),-},\qquad 
E^{\pm}_{\xx} = h+E_{\uu}^{(v),\pm}.$$
Using exactly the same limits computed in the proof of Proposition \ref{prop:limittoRpoly}, we have 
$$\lim_{e^{v\alpha_i}\to 0}
(-y)^{|J_{\vv\uu}^-|}
\prod_{\alpha>0}\frac{1+y{e^{vw\alpha}}}{1-e^{vw\alpha}}B_{\vv\uu}
= 
\begin{cases}
(-y)^{|J_{\uu}^{(v),-}|}((-y)-1)^{|E_{\uu}^{(v)}|}, & E_{\uu}^{(v),-}=\varnothing;\\
0, & \text{otherwise}.
\end{cases}$$
This concludes the proof.
\end{proof}

\begin{example}
Consider the case $A_2$, $u=s_2$ and $w=s_1s_2$.
We have 
$$\SMC_y(Y(u)^\circ)|_w \cdot \prod_{\alpha>0}\frac{1+ye^{w\alpha}}{1-e^{w\alpha}}
=
\frac{(1+y)(-ye^{\alpha_2} - 1)}{
(e^{\alpha_1}-1)(e^{\alpha_2}-1)} .$$
Then we can compute the limits as follows
\begin{align*}
v=\operatorname{id} & \qquad 
e^{\alpha_1}\to 0,e^{\alpha_2}\to 0 &
\frac{(1+y)(0-1)}{(0-1)(0-1)} &= -y-1,\\
v=s_1 &\qquad 
e^{\alpha_1}\to \infty,e^{\alpha_1+\alpha_2}\to 0 & 
\frac{(1+y)(0-1)}{(\infty -1)(0-1)} &=0,\\
v=s_2 &\qquad 
e^{\alpha_1+\alpha_2}\to 0,e^{\alpha_2}\to \infty & 
\frac{(1+y)(-y\infty-1)}{(0 -1)(\infty-1)} &=y^2-y,\\
v=s_1s_2 &\qquad 
e^{\alpha_2}\to 0,e^{\alpha_1+\alpha_2}\to \infty & 
\frac{(1+y)(0-1)}{(\infty -1)(0-1)} &=0,\\
v=s_2s_1 &\qquad 
e^{\alpha_1+\alpha_2}\to \infty,e^{\alpha_1}\to 0 & 
\frac{(1+y)(-y\infty-1)}{(0-1)(\infty-1)} &=y^2-y,\\
v=s_1s_2s_1 &\qquad 
e^{\alpha_2}\to \infty,e^{\alpha_1}\to \infty & 
\frac{(1+y)(-y\infty-1)}{(\infty-1)(\infty-1)} &=0.
\end{align*}
\end{example}

\section{Combinatorial formula in type $A$}\label{sec:comb}

\newcommand{\PD}[2][1pc]{%
\setlength{\unitlength}{#1}
\def\BPDframe{%
    \thinlines%
    \color{lightgray}%
    \put(0,0){\line(0,1){1}}%
    \put(1,0){\line(0,1){1}}%
    \put(0,0){\line(1,0){1}}%
    \put(0,1){\line(1,0){1}}%
    \linethickness{0.08\unitlength}
    \color{teal}}
\def\O{
\begin{picture}(1,1)
    \BPDframe
\end{picture}}
\def\X{
\begin{picture}(1,1)
    \BPDframe
    \qbezier(0.5,0)(0.5,0.5)(0.5,1)
    \qbezier(0,0.5)(0.5,0.5)(1,0.5)    
\end{picture}}
\def\B{
\begin{picture}(1,1)
    \BPDframe
    \qbezier(0.5,0)(0.5,0.5)(1,0.5)
    \qbezier(0.5,1)(0.5,0.5)(0,0.5)
\end{picture}}
\def\b{
\begin{picture}(1,1)
    \BPDframe
    \qbezier(0.5,0)(0.5,0.5)(0,0.5)
    \qbezier(0.5,1)(0.5,0.5)(1,0.5)
\end{picture}}
\def\D{
\begin{picture}(0,1)
    \color{green}
    \linethickness{0.1\unitlength}
    \qbezier(0,0)(0,0.1)(0,0.1)
    % \qbezier(0,0.15)(0,0.35)(0,0.35)
    \qbezier(0,0.4)(0,0.6)(0,0.6)
    % \qbezier(0,0.65)(0,0.85)(0,0.85)
    \qbezier(0,0.9)(0,1)(0,1)
\end{picture}
}
\def\M##1{\begin{picture}(1,1)%
    \put(0,0.2){\makebox[\unitlength]{\(##1\)}}
\end{picture}}
\begin{array}{@{\,}c@{\,}}
{\def\arraystretch{0}
\setlength{\arraycolsep}{0pc}
\color{teal}
\begin{array}{@{}l@{}}%
#2\end{array}}
\end{array}}

In this section, we give a $K$-theoretic version of the results in \cite[Section 7]{FGSX}.
Following \cite[Section 7.3]{FGSX}, let us consider all possible $n$-periodic tilings on the square grid $ \{1,\ldots,k\}\times \mathbb{Z}$ using tiles
$$\PD[1.5pc]{\B}\qquad \PD[1.5pc]{\X}$$
Here, $n$-periodicity means the tile at $(i,j)$ is the same as that at $(i+n,j)$. Let $\tilde{S}_n$ be the extended affine Weyl group of type $A_{n-1}$.
We can read an affine permutation $f\in \tilde{S}_n$ from a $n$-periodic pipe dream such that the $i$-th pipe on the bottom is connected to the $f(i)$-th pipe on the top.
For example, when $n=7$, $k=3$, 
the following tiling 
$$\PD{
% \M{}\D
% \M{\text{-}6}\M{\text{-}5}\M{\text{-}4}\M{\text{-}3}\M{\text{-}2}\M{\text{-}1}\M{0}
% \D\M{1}\M{2}\M{3}\M{4}\M{5}\M{6}\M{7}
% \D\M{8}\M{9}\M{1\kern -0.1em0}\M{1\kern -0.1em1}\M{1\kern -0.1em2}\M{1\kern -0.1em3}\M{1\kern -0.1em4}\D\\
\M{}
\D\M{}\M{}\M{}\M{}\M{}\M{}\M{}
\D\M{}\M{}\M{}\M{}\M{}\M{}\M{}
\D\M{}\M{}\M{}\M{}\M{}\M{}\M{}
\D
\\
\M{}\D
    \X\X\B\X\X\B\X
\D\X\X\B\X\X\B\X\D\X\X\B\X\X\B\X\D\\
\M{\cdots\qquad\qquad}\D
    \B\B\X\B\X\B\X
\D\B\B\X\B\X\B\X\D\B\B\X\B\X\B\X\D\M{\qquad\qquad\cdots}\\
\M{}\D
    \X\B\B\X\B\B\X
\D\X\B\B\X\B\B\X\D\X\B\B\X\B\B\X\D
\\
\M{}\D
\M{\text{-}6}\M{\text{-}5}\M{\text{-}4}\M{\text{-}3}\M{\text{-}2}\M{\text{-}1}\M{0}
\D\M{1}\M{2}\M{3}\M{4}\M{5}\M{6}\M{7}
\D\M{8}\M{9}\M{1\kern -0.1em0}\M{1\kern -0.1em1}\M{1\kern -0.1em2}\M{1\kern -0.1em3}\M{1\kern -0.1em4}\D\\
}$$
is a pipe dream  with  reading affine permutation given by
$$f(1)=2,\,\, f(2)=6,\,\, f(3)=5,\,\, f(4)=10,\,\, f(5)=8,\,\, f(6)=11,\,\, f(7)=7.$$
Let us denote by $\mathsf{PD}(f)$ the set of all such tilings with reading affine permutation $f$. 
Let $x_1,\ldots,x_k,t_1,\ldots,t_n$ be variables. 
Let us define the weight of each tile. Consider the tile at $(i,j)$
$$\operatorname{wt}\left(\PD{
\M{}\M{a}\\
\M{a}\B\M{b}\\
\M{}\M{b}}\right)
= \begin{cases}
\dfrac{(1+y)x_i}{t_j+yx_i}, & \textit{ if }
a<b,\\[2ex]
\dfrac{(1+y)t_j}{t_j+yx_i}, & \textit{ if }
a>b.\\
\end{cases}\qquad 
\operatorname{wt}\left(\PD{
\M{}\M{b}\\
\M{a}\X\M{a}\\
\M{}\M{b}}\right)
= \begin{cases}
\dfrac{t_j-x_i}{t_j+yx_i}, & \textit{ if } 
a<b,\\[2ex]
-y\dfrac{t_j-x_i}{t_j+yx_i}, & \textit{ if } 
a>b,\\
\end{cases}$$
where $a$ (resp. $b$) indicates the pipe is from the $a$-th (resp. $b$-th) column of the bottom. 
For each $\pi\in \mathsf{PD}(f)$, we define the weight $\operatorname{wt}(\pi)$ to be the product of weights of tiles in a period, i.e. $\{1,\ldots,k\}\times \{1,\ldots,n\}$. 
We define 
$$\tilde{G}_f = \sum_{\pi\in \mathsf{PD}(f)} \operatorname{wt}(\pi). $$
In view of \cite{SZ2025}, this could be viewed as a generalization of stable affine Grothendieck polynomial. 

The maximal torus $T\subset G=GL_n$ is $(\mathbb{C}^*)^n$. Let $t_i$ be the character of projecting to the $i$-th component. Then $K_T(\pt)=\mathbb{Z}[t_1^{\pm 1},\cdots, t_n^{\pm 1}]$, and 
% For $G=GL_n$, we can identify 
% $$X_*(T) = 
% \mathbb{Z}\epsilon_1\oplus \cdots \oplus \mathbb{Z}\epsilon_n
% \supset 
% \mathbb{Z}(\epsilon_1-\epsilon_2)\oplus 
% \cdots \oplus 
% \mathbb{Z}(\epsilon_{n-1}-\epsilon_n)=Q^\vee. $$
% Let $y_i= e^{\epsilon_i}\in K_T(pt)$ \red{CS: $\epsilon_i$ is a cocharacter, here we want a character. Perhaps we should also use the notation $t_i$ from above.}. 
$$e^{\alpha_1}=t_1/t_2,\ldots,
e^{\alpha_{n-1}}=t_{n-1}/t_n,e^{\alpha_0} = t_{n}/t_1. $$
Let us consider the Grassmannian $\operatorname{Gr}_k(\mathbb{C}^n)$. The open projected Richardson varieties are known as {open positroid varieties}, which are indexed by bounded affine permutations \cite{KLS13}: 
$$\mathcal{B}
=\left\{
\begin{array}{c}
\text{bijections $f\colon\mathbb{Z}\to \mathbb{Z}$}
\end{array}:
\begin{array}{c}
f(i+n)=f(i)+n\\[1ex]
\frac{1}{n}\sum_{i=1}^n(f(i)-i)=k\\[1ex]
i\leq f(i)\leq i+n
\end{array}
\right\}.$$
% Let 
% $$\lambda = \epsilon_1+\cdots+\epsilon_k.$$
% Then $G/P$ is the Grassmannian $\operatorname{Gr}_k(\mathbb{C}^n)$. 
Let $x_1,\ldots,x_k$ be the formal K-theoretic Chern roots of the tautological subbundle $\mathcal{V}$ on $\operatorname{Gr}_k(\mathbb{C}^n)$, then one can view $K_T(\Gr_k(\bbC^n))$ as a quotient of 
\[
 \bbZ[t_1^{\pm 1}, ...,t_n^{\pm 1}][ x_1^{\pm 1},...,x_k^{\pm 1}]^{S_k}. 
\]
The following theorem shows that $\tilde G_f$ is a polynomial representative of the SMC class of the open projected Richardson variety. 

\begin{thm}
For any $f\in \mathcal{B}$, the function $\tilde{G}_f$ is symmetric in $x_1,\ldots,x_k$ and 
$$\SMC_y(\mathring{\Pi}_f) = \tilde{G}_f\in K_T(\operatorname{Gr}_k(\mathbb{C}^n))[[y]].$$
\end{thm}
\begin{proof}
The proof is completely parallel to the one in \cite[Section 7]{FGSX}, and we only sketch it. 
The key step is an analog of \cite[Theorem 7.3]{FGSX}, a diagrammatic formula for the localization of $\SMC_y(\mathring{\Sigma}^f)|_t$ for all $t,f\in \tilde{S}_n$. 
The remaining argument is based on Theorem \ref{thm:main} and a diagrammatic calculus. 
The major difference is the weight 
$$\def\midarrow{\tikz{\draw[thick,->] (-0.001,0) -- +(0.001,0);}}
\def\midlabel#1{\tikz{%
    \draw[fill=white,white] (0,0) circle [radius=.15];
    \node at (0,0){\small #1};}}
\operatorname{wt}\left(\begin{matrix}
\begin{tikzpicture}[scale=1.5] 
\draw [thick,->] (0,0) node [left]{$u$}
-- 
node[pos=0.2]{\midlabel{$x$}} 
node[pos=0.8]{\midlabel{$b$}} 
(1,1) node [left]{$u$}; 
\draw [thick,->] (1,0) node [left]{$v$}
--
node[pos=0.2]{\midlabel{$y$}} 
node[pos=0.8]{\midlabel{$a$}} 
 (0,1) node [left]{$v$}; 
\end{tikzpicture}
\end{matrix}\right)
=\frac{1}{v+yu}\begin{cases}
(1+y)u& a=x<b=y,\\
(1+y)v& a=x>b=y,\\
v-u& a=y>b=x,\\
(-y)(v-u)& a=y<b=x.\\
\end{cases}$$
Using the same notation as in the proof of Theorem 7.3 in \textit{loc.cit.} (with $y_i$ changed to $t_i$) we have, 
if $f>s_if$
\def\wt{\operatorname{wt}}
\def\t{t}
\begin{align*}
\wt(\mathcal{D}_{s_i \t},\beta_{f})
& =
\frac{(1+y)t_{i}}{t_{i}+yt_{i+1}}
s_i\big(\wt(\mathcal{D}_{\t},\beta_{f})\big)
+
\frac{t_{i}-t_{i+1}}{t_{i}+yt_{i+1}}
s_i\big(\wt(\mathcal{D}_{\t},\beta_{s_if})\big)\\
& = 
\frac{(1+y)e^{\alpha_i}}{e^{\alpha_i}+y}
s_i\big(\wt(\mathcal{D}_{\t},\beta_{f})\big)
+
\frac{e^{\alpha_i}-1}{e^{\alpha_i}+y}
s_i\big(\wt(\mathcal{D}_{\t},\beta_{s_if})\big),
\end{align*}
and if $f>s_if$

\begin{align*}
\wt(\mathcal{D}_{s_i \t},\beta_{f})
& =
\frac{(1+y)t_{i+1}}{t_{i}+yt_{i+1}}
s_i\big(\wt(\mathcal{D}_{\t},\beta_{f})\big)
+(-y)
\frac{t_{i}-t_{i+1}}{t_{i}+yt_{i+1}}
s_i\big(\wt(\mathcal{D}_{\t},\beta_{s_if})\big)\\
& = 
\frac{(1+y)}{e^{\alpha_i}+y}
s_i\big(\wt(\mathcal{D}_{\t},\beta_{f})\big)
+(-y)
\frac{e^{\alpha_i}-1}{e^{\alpha_i}+y}
s_i\big(\wt(\mathcal{D}_{\t},\beta_{s_if})\big).
\end{align*}
This agrees with the recursion for $\SMC_y(\mathring{\Sigma}^f)|_t$ from Corollary \ref{cor:recSMCaff}. 
\end{proof}

\begin{example} Let $k=1$, and consider 
$X:=\operatorname{Gr}_1(\mathbb{C}^n)=\mathbb{P}^{n-1}$. 
% Recall the Euler exact sequence
% $$0\longrightarrow T_X^*\longrightarrow \mathbb{C}^n\otimes \mathcal{O}_X(-1)\longrightarrow \mathcal{O}_X\longrightarrow 0. $$
% We have 
% $$\lambda_y(T_X^*)
% =\frac{\lambda_y(\mathbb{C}^n\otimes \mathcal{O}_X(-1))}{\lambda_{y}(\mathcal{O}_X)}
% =\frac{\prod_{j=1}^n\big(1+y\frac{x}{t_i}\big)}{1+y}.$$
Let $L_A=\mathbb{P}(\mathbb{C}^A)$ be the coordinate plane for a subset $A\subseteq \{1,\ldots,n\}$. 
Let $L_A^\circ\subset L_A$ be such that 
$L_A = \bigsqcup_{B\subset A} L_B^\circ$.

There is a bijection between $ \mathcal{B}$ subsets of $\{1,...,n\}$ given by $f\mapsto  A_f:=\{1\leq i\leq n:f(i)\neq i\}$, and we have 
$\mathring{\Pi}_f = L^\circ_{A_f}$. Denote $A=A_f$ and $L=L_A$, and we compute $\SMC_y(\mathring{\Pi}_f)$ geometrically. 

The Euler exact sequence over $L$ 
$$0\longrightarrow T_L^*\longrightarrow \mathbb{C}^A\otimes \mathcal{O}_L(-1)\longrightarrow \mathcal{O}_L\longrightarrow 0$$
gives 
$$\lambda_y(T_{L}^*)
= \frac{\lambda_y(\mathbb{C}^A\otimes \mathcal{O}_L(-1))}{\lambda_y(\mathcal{O}_L)}
= \frac{1}{1+y}\prod_{j\in A} \left(1+y\frac{x}{t_j}\right). $$

The structure sheaf of $L=L_A$ can be represented by Koszul complex 
$$\bigotimes_{j\notin A} [\mathcal{O}(-1)\stackrel{x_j}\longrightarrow \mathcal{O}_X] = 
\prod_{j\notin A}\left(1-\frac{x}{t_j}\right).$$
Thus,
$$\MC_y(L_A) = 
\frac{1}{1+y}\cdot 
\prod_{j\in A}\left(1+y\frac{x}{t_j}\right)\cdot 
\prod_{j\notin A}\left(1-\frac{x}{t_j}\right)=
\frac{1}{1+y}\cdot 
\prod_{j\in A}\left(1+\frac{(1+y)\frac{x}{t_j}}{1-\frac{x}{t_j}}\right)\cdot 
\prod_{j}\left(1-\frac{x}{t_j}\right),$$
and
%By noticing 
% $\left(1+y\frac{x}{t_j}\right)=\left(1-\frac{x}{t_j}+(1+y)\frac{x}{t_j}\right)$, 
% it is easy to see 
$$\MC_y(\mathring{\Pi}_f)
=\MC_y(L_A^\circ) = 
\frac{1}{1+y}\cdot 
\prod_{j\in A}\left((1+y)\frac{x}{t_j}\right)\cdot 
\prod_{j\notin A}\left(1-\frac{x}{t_j}\right).$$
As a result, 
$$\SMC_y(\mathring{\Pi}_f) = \frac{\operatorname{MC}_y(\mathring{\Pi}_f)}{\lambda_y(T_X^*)}
= 
\prod_{j\in A}\frac{(1+y)\frac{x}{t_j}}{1+y\frac{x}{t_j}}\cdot 
\prod_{j\notin A}\frac{1-\frac{x}{t_j}}{1+y\frac{x}{t_j}}
= 
\prod_{j\in A}\frac{(1+y)x}{t_j+yx}\cdot 
\prod_{j\notin A}\frac{t_j-x}{t_j+yx}.$$

Combinatorially, note that for the tiling of $f\in \calB$,  the $\PD{\B}$-tiles are put exactly at the $(1,j)$-entry for $j\in A$. 
For example, when $A=\{2,4,5,6,9\}\subset \{1,\ldots,9\}$, the unique $f\in \mathcal{B}$ and the unique element in $\mathsf{PD}(f)$ can be represented as 
$$\PD{
\D\M{\text{-}1}\M{2}\M{1}\M{4}\M{5}\M{6}\M{3}\M{7}\M{9}\D\\
\D\B\X\B\X\X\X\B\B\X\D\\
\D\M{1}\M{2}\M{3}\M{4}\M{5}\M{6}\M{7}\M{8}\M{9}\D\\
}$$
We thus have 
$$\tilde{G}_f = \prod_{j\in A}\frac{(1+y)x}{t_j+yx}\cdot 
\prod_{j\notin A}\frac{t_j-x}{t_j+yx}.$$
This agrees with the computation above.

% Let us consider 
% $f\in \tilde{S}_n$ defined by 
% $$f(j)=j+\delta_{ij}n,\qquad j=1,\ldots,n.$$
% Then $\mathring{\Pi}_f$ is the torus fixed point corresponding to the $i$-th coordinate space. Its structure sheaf can be represented by a Koszul complex
% $$\bigotimes_{j\neq i}[\mathcal{O}(-1)\stackrel{x_i}\longrightarrow \mathcal{O}_X] = \prod_{j\neq i}\left(1-\frac{x}{t_j}\right).$$
% As a result, 
% $$\SMC_y(\mathring{\Pi}_f) = \frac{\operatorname{MC}_y(\mathring{\Pi}_f)}{\lambda_y(T_X^*)}
% = 
% \frac{(1+y)}{1+y\frac{x}{t_i}}
% \prod_{j\neq i}\frac{1-\frac{x}{t_j}}{1+y\frac{x}{t_j}}
% =
% \frac{(1+y)t_i}{t_i+yx}
% \prod_{j\neq i}\frac{t_i-x}{t_i+yx}. $$
% Combinatorially, there is exactly one element of $\mathsf{PD}(f)$ which has a unique $\PD{\B}$ at the $i$-th column in an peroid. For example, when $n=4$ and $i=3$
% $$\PD{
% \M{}\M{}\D\M{}\M{}\M{}\M{}
% \D\M{}\M{}\M{}\M{}\D\M{}\M{}\M{}\M{}\D\\
% \M{\cdots}\M{}\D\X\X\B\X\D\X\X\B\X\D\X\X\B\X\D\M{}
% \M{\cdots}\\
% \M{}\M{}\D\M{\text{-}3}\M{\text{-}2}\M{\text{-}1}\M{0}
% \D\M{1}\M{2}\M{3}\M{4}
% \D\M{5}\M{6}\M{7}\M{8}\D}
% $$
% It is easy to see the combinatorial formula agrees with the computation above. 
% More general, for a subset $J$, we can consider the coordinate subspace $W_J\subseteq X$ of $J$. 
% The the structure sheaf of $P$ can be represented by 
% $$\bigotimes_{j\notin J} [\mathcal{O}(-1)\stackrel{x_i}\longrightarrow \mathcal{O}_X] = 
% \prod_{j\notin J}\left(1-\frac{x}{t_j}\right).$$
% As a result, 
% $$\MC_y(P)=\frac{\prod_{j\in J}(1+y\frac{x}{t_j})}{1+y}\cdot \prod_{j\notin J}\left(1-\frac{x}{t_j}\right). $$
\end{example}

\bibliographystyle{alpha}
\bibliography{MCRichardson}

\end{document}